\documentclass[12pt,a4paper,reqno]{amsart}

\usepackage{amssymb}
\usepackage{tabu}
\usepackage[english]{babel}
\usepackage[utf8]{inputenc}
\usepackage[font=small,labelfont=bf]{caption}
\usepackage[headings]{fullpage}
\usepackage[colorlinks]{hyperref}

\DeclareMathOperator{\ad}{ad}
\DeclareMathOperator{\Aut}{Aut}
\DeclareMathOperator{\Der}{Der}
\DeclareMathOperator{\End}{End}
\DeclareMathOperator{\diag}{diag}
\DeclareMathOperator{\GL}{GL}
\DeclareMathOperator{\Heis}{Heis}
\DeclareMathOperator{\I}{Isom}
\DeclareMathOperator{\id}{id}
\DeclareMathOperator{\OO}{O}
\DeclareMathOperator{\SO}{SO}
\DeclareMathOperator{\spann}{span}
\DeclareMathOperator{\SU}{SU}
\DeclareMathOperator{\Sym}{Sym}
\DeclareMathOperator{\tr}{tr}
\DeclareMathOperator{\U}{U}

\newcommand{\heis}{\mathfrak{heis}}
\newcommand{\hh}{\mathfrak{h}}
\newcommand{\so}{\mathfrak{so}}

\theoremstyle{plain}
\newtheorem{theorem}{Theorem}[section]
\newtheorem{lemma}[theorem]{Lemma}
\newtheorem{corollary}[theorem]{Corollary}
\newtheorem{proposition}[theorem]{Proposition}

\theoremstyle{remark}
\newtheorem{remark}[theorem]{Remark}

\numberwithin{equation}{section}

\begin{document}

\title[Left-invariant metrics on six-dimensional nilpotent Lie
groups]{The moduli space of left-invariant metrics of a class of
  six-dimensional nilpotent Lie groups}

\dedicatory{Dedicated to Carlos Olmos on the occasion of his 60th
  birthday}

\author{Silvio Reggiani}
\address{CONICET and Universidad Nacional de Rosario, ECEN-FCEIA,
  Departamento de Ma\-te\-má\-ti\-ca. Av. Pellegrini 250, 2000
  Rosario, Argentina.}
\email{\href{mailto:reggiani@fceia.unr.edu.ar}{reggiani@fceia.unr.edu.ar}}
\urladdr{\url{http://www.fceia.unr.edu.ar/~reggiani}}

\author{Francisco Vittone}
\address{CONICET and Universidad Nacional de Rosario, ECEN-FCEIA,
  Departamento de Ma\-te\-má\-ti\-ca. Av. Pellegrini 250, 2000
  Rosario, Argentina.}
\email{\href{mailto:vittone@fceia.unr.edu.ar}{vittone@fceia.unr.edu.ar}}
\urladdr{\url{http://www.fceia.unr.edu.ar/~vittone}}

\date{\today}

\thanks{Supported by CONICET. Partially supported by SeCyT-UNR}



\maketitle

\begin{abstract}
  In this paper we determine the moduli space, up to isometric
  automorphism, of left-invariant metrics on a $6$-dimensional Lie
  group $H$, such that its Lie algebra $\mathfrak{h}$ admits a complex
  structure and has first Betti number equal to four. We also
  investigate which of these metrics are Hermitian and classify the
  corresponding complex structures.
\end{abstract}

\section{Introduction}

The present work concerns the study of the moduli space of
left-invariant metrics on nilpotent Lie groups up to
diffeomorphism. As it was proved by Wolf in \cite{wolf-1963}, and
later generalized in \cite{alekseevskii-1971,gordon-wilson-1988} for
Riemannian solvmanifolds, this is equivalent to the study of the
moduli space of left-invariant metrics up to isometric
automorphism. Notice that by Mal'cev criterion, every compact
nilmanifold is the quotient of a simply connected nilpotent Lie group
by a discrete subgroup. So, the problem we approach is closely related
with the problem of determining the moduli space of invariant metrics
on compact nilmanifolds, up to diffeomorphism. There is some previous
work addressing the problem of determining the moduli space of
left-invariant metrics on Lie groups. One can mention the work of
Ha-Lee \cite{ha-lee-2009} which solves the problem in dimension 3. In
\cite{lauret-2003}, Lauret classifies all Lie groups with only one
left-invariant metric up to isometry and scaling. Later, in
\cite{kodama-takahara-tamaru-2011}, Kodama, Takahara and Tamaru give
another proof of Lauret's theorem and also study the case when the
moduli space up to isometry and scaling has dimension $1$. In
\cite{Di_Scala_2012}, Di Scala classifies the moduli space of
left-invariant metrics up to isometric isomorphism on the Iwasawa
manifold.  For semi-definite metrics, some progress has also been
made, for instance Vukmirović classifies in \cite{vukmirovic-2015} the
pseudo-Riemannian left-metrics metrics on Heisenberg groups. In a
recent paper \cite{kondo2020classification}, Kondo and Tamaru
determine the moduli space up to isometry and scaling of Lorentzian
left-invariant metrics on certain nilpotent Lie groups. Some of the
above results use the so-called Milnor frames, a concept that dates
back to the famous paper of Milnor \cite{milnor-1976}, for describing
the Lie bracket on $3$-dimensional Lie groups in terms of an
orthonomal basis. In \cite{hashinaga2016}, it is showed a closed
retationship between the moduli space of left-invariant metrics and
Milnor frames by giving several examples of Milnor-type theorems.

Geometric structures associated to low-dimensional Lie groups with
left-invariant metrics have been widely studied.  For the case of
$6$-dimensional nilpotent Lie groups, particular attention has been
paid to the Iwasawa manifold $\mathcal{I}=\Gamma\backslash H$, which
is a compact quotient of the $3$-dimensional complex Heisenberg group
$H$. The Hermitian geometry of $\mathcal{I}$, with a standard metric,
was studied in \cite{abbena-1997,abbena-2001} and
\cite{Ketsetzis_2004}. The classification of Di Scala
\cite{Di_Scala_2012}, of the moduli space of metrics up to isometric
automorphism, relies on fixing a distinguished complex structure on
the Lie algebra $\mathfrak{h}$ of $H$, which allows to determine the
automorphism group $\Aut(\mathfrak{h})$ in an elegant way on the
canonical basis.

In \cite{salamon-2001}, Salamon classified all $6$-dimensional Lie
algebras $\mathfrak{g}$ which admit a complex structure. Such Lie
algebras are grouped according to the first Betti number of
$\mathfrak{g}$. In particular, $3$-dimensional complex Heisenberg Lie
algebra belongs to the class whose first Betti number is equal to
$4$. This class contains five Lie algebras that are characterized by
the property that $\textrm{dim}[\mathfrak{g},\mathfrak{g}]=2$. In the
notation of \cite{salamon-2001}, they are $\hh_2, \hh_4, \hh_5, \hh_6$
and $\hh_9$ (see Section \ref{nildim6}). The complex Heisenberg Lie
algebra is $\hh_5$. All these Lie algebras are $2$-step nilpotent,
with the exception of $\hh_9$ which is $3$-step nilpotent.

As a natural continuation of the work in \cite{Di_Scala_2012}, in this
paper we deal with $6$-dimensional, simply connected, nilpotent Lie
groups $H$ which admit a left-invariant complex structure and their
Lie algebras $\mathfrak{h}$ have first Betti number equal to $4$. Our
main goal is to classify the moduli space of left-invariant metrics,
up to isometric automorphism, for this particular
family. Left-invariant metrics on $H$ are in a one-to-one
correspondence with the inner products on $\mathfrak{h}$. Moreover,
$\Aut(H)$ is isomorphic to $\Aut(\mathfrak{h})$ and the classification
of left-invariant metrics on $H$ up to automorphism reduces to the
classification of inner products of $\mathfrak{h}$, up to an
automorphism of $\mathfrak{h}$.

It is important to observe that the methods developed in
\cite{Di_Scala_2012} cannot be directly adapted to any of the other
Lie algebras studied here. In fact, while $\Aut(\hh_5)$ is a complex
Lie group, this is not true for any of the other Lie algebras we are
dealing with. For each case, we explicitly find the automorphism group
$\Aut(\mathfrak{h})$ (for the case of $\hh_4$ and $\hh_5$ this was
also done in \cite{magnin-2007} by means of computational methods, and
in \cite{saal-1996} for $H$-type Lie algebras). In this way, we are
able to describe the moduli space $\mathcal M(H)/{\sim}$ of
left-invariant metrics on $H$ up to isometric automorphisms. It is
interesting to notice that the only case in which
$\mathcal M(H)/{\sim}$ is a differentiable manifold is when the Lie
algebra of $H$ is $\hh_9$. Our classification also shows that the
moduli space up to isometry and scaling of the Lie algebra
$\mathfrak h_6$ has dimension $1$. This extends the result in
\cite{kodama-takahara-tamaru-2011}, as the Lie algebra $\mathfrak h_6$
is not present among the examples there of Lie groups with
$1$-dimensional moduli space up to isometry and scaling. We also
obtain the full isometry group $\I(H, g)$ associated to each
left-invariant metric $g$ on $H$.

Another interesting problem is to determine which of the metrics $g$
are Hermitian, that is, when there exists an invariant complex
structure $J$ on $H$ such that $(g, J)$ is an Hermitian structure. This
turns out to be a very difficult computational problem. Even though
the complex structures on $6$-dimensional nilpotent Lie algebras $\hh$
were classified in \cite{ceballos-2014}, it is very difficult to
explicitly obtain the form of a particular complex structure $J$ on a
given basis of $\mathfrak{h}$. For the case of the Iwasawa manifold,
the set $\mathcal{C}$ of complex structures compatible with a standard
metric and orientation of $\mathcal{I}$ was described in
\cite{abbena-2001} by means of topological methods. The authors show
there that $\mathcal{C}$ is the disjoint union of the standard complex
structure $J_0$ and a $2$-sphere.

In this paper we give a complete classification of the Hermitian
structures for $\mathfrak h_4, \mathfrak h_5$ and $\mathfrak h_6$. The
problem becomes wild for $\mathfrak h_2$ and $\mathfrak h_9$, however
we obtain some interesting partial results. For $\mathfrak h_2$ we prove
that every left-invariant metric admits a finite number of compatible
complex structures, and for $\mathfrak h_9$ we include a qualitative
analysis and prove that there are left-invariant metrics which are
not Hermitian.

We hope that the methods developed here will be useful to study
the remaining cases in the classification of \cite{salamon-2001}.

The authors would like to thank Antonio Di Scala for suggesting the
problem and for very useful discussions about it. They also thank
Isolda Cardoso for her suggestions which helped to simplify some
computations.

\section{Preliminaries}

\subsection{The moduli space of left-invariant metrics}

Let $H$ be a simply connected Lie group with Lie algebra $\mathfrak
h$. Every left-invariant metric on $H$ is uniquely determined by a
(positive definite) inner product on $\mathfrak h$, so, the set
$\mathcal M(H)$ of left-invariant metrics on $H$ is identified, after
the choice of a basis of $\mathfrak h$, with the symmetric space
$\Sym_n^+ = \GL_n(\mathbb R)/{\OO(n)}$, where $n = \dim H$. Recall
that the group $\Aut(H)$ of automorphisms of $H$ acts on the right on
$\mathcal M(H)$ by
\begin{equation}
  \label{eq:24}
  g \cdot \varphi = \varphi_*(g),
\end{equation}
for $g \in \mathcal M(H)$ and $\varphi \in \Aut(H)$, where
$\varphi_*(g)(u,v)=g(d\varphi(u),d\varphi(v))$, i.e, $\varphi_*(g)$ is
the pullback of $g$ by $\varphi$. The moduli space of left-invariant
metrics of $H$ up to isometric automorphisms is
$\mathcal M(H)/{\sim}$, where $\sim$ is the equivalence relation
induced by the action given in (\ref{eq:24}). Since $H$ is simply
connected, $\Aut(H)$ is isomorphic to the group $\Aut(\mathfrak h)$ of
automorphisms of its Lie algebra $\mathfrak h$, which we can identify
with a subgroup of $\GL_n(\mathbb R)$. If we think of $\Sym_n^+$ as
the set of symmetric positive definite matrices of size $n \times n$,
then the action of $\Aut(H)$ on $\mathcal M(H)$ is equivalent to the
action of $\Aut(\mathfrak h)$ on $\Sym_n^+$ given by
\begin{equation*}
  X \cdot A = A^TXA,
\end{equation*}
for $X \in \Sym_n^+$, $A \in \Aut(\mathfrak h)$.

\subsection{Complex structures}

In the same spirit as in the previous paragraphs, one can identify the
set $\mathcal C(H)$ of left-invariant complex structures on $H$ with
\begin{equation*}
\mathcal C(\mathfrak h) = \{J \in \End_{\mathbb R}(\mathfrak h): J^2 =
-\id_{\mathfrak h} \text{ and } N_J = 0\}
\end{equation*}
where $N_J$ is the so-called Nijenhuis tensor of $J$, which is given
for $X, Y \in \mathfrak h$ by
\begin{equation*}
  N_J(X, Y) = [JX, JY] - J[JX, Y] - J[X, JY] - [X, Y].
\end{equation*}

In the same manner, left-invariant abelian structures on $H$ are
identify with the subset
\begin{equation*}
\mathcal A(\mathfrak h) = \{J \in \mathcal C(\mathfrak h): [JX, JY] =
[X, Y] \text{ for all } X, Y \in \mathfrak h\}
\end{equation*}
of $\mathcal{C}(\mathfrak h)$.  We say that two complex (resp.\
abelian) structures are equivalent if they are conjugated by an
element of $\Aut(\mathfrak h)$. It is customary to consider the
left-action of $\Aut(\mathfrak h)$ on $\mathcal C(\mathfrak h)$, which
is given by
\begin{equation*}
  \varphi \cdot J = \varphi J \varphi^{-1}
\end{equation*}
for $J \in \mathcal C(\mathfrak h)$ and
$\varphi \in \Aut(\mathfrak h)$. Recall however that the pullback
action of $\Aut(H)$ on $\mathcal C(H)$ induces the right-action of
$\Aut(\mathfrak h)$ on $\mathcal C(\mathfrak h)$ given by
$J \cdot \varphi = \varphi^{-1} J \varphi$. These two actions have the
same orbits and leave $\mathcal A(\mathfrak h)$ invariant. This is not
true, in general, for the left- and right-actions of
$\Aut(\mathfrak h)$ on $\Sym_n^+$.

\subsection{Nilpotent Lie algebras of dimension $6$}
\label{nildim6}

In this section we shall recall some relevant notation and useful
properties of $6$-dimensional nilpotent Lie algebras which will be
used in the whole paper. For further details we refer the reader to
\cite{salamon-2001}.  Let $\hh$ be a $6$-dimensional Lie algebra,
$\mathcal{B} = \{e_1, \ldots, e_6\}$ a basis of $\hh$ and
$\mathcal{B}^* = \{e^1, \ldots, e^6\}$ the dual basis of $\hh^*$. For
each $i = 1, \ldots, 6$, we write
\begin{equation*}
  de^k = \sum_{i<j} c^k_{ij} \, e^{ij},
\end{equation*}
where $e^{ij}$ denotes the exterior product $e^i \wedge e^j$. Since
$\hh$ is nilpotent and $6$-dimensional), one can choose $\mathcal{B}$
in such a way that $c^{k}_{ij} \in \{0,1\}$ for
$i, j, k \in \{1, \ldots, 6\}$ and such that $c^k_{ij}=0$ for
$i, j < k$.  In this way, one can completely determine $\hh$ by
knowing the differentials
\begin{equation*}
de^1, de^2, \ldots, de^6
\end{equation*}
since this information together with the formula
$d\theta(X, Y) = - \theta([X, Y])$, for
$\theta \in \Lambda^1(\mathfrak h)$, allow us to reconstruct all the
Lie brackets of $\hh$.  Following Salamon's notation, if
$de^k = e^{i_1 j_1}+ \cdots + e^{i_lj_l}$ we shall simply denote it by
$i_1j_1+\cdots+ i_lj_l$. In this way, for example, we will write
$$\hh = (0, 0, 0, 0, 0, 12 + 34)$$ for the Lie algebra that admits a
basis $\mathcal{B}$ such that $de^6 = e^{12} + e^{34}$, and
hence on which the only non trivial brackets are
$[e_1,e_2]=[e_3,e_4]=-e_6$.

As we indicated in the Introduction, we are interested on those
$6$-dimensional Lie algebras which admit a complex structure and has
their first Betti number equal to $4$. These are the Lie algebras
which, in the classification of Salamon, belong to the same class of
the Lie algebra of the Iwasawa manifold. With the notation presented
above, there are exactly five $6$-dimensional nilpotent Lie algebras
with these properties:
\begin{equation}\label{liealgs}
  \begin{array}{l}
    \mathfrak{h}_2 = (0, 0, 0, 0, 12, 34)\\
    \mathfrak{h}_4 = (0, 0, 0, 0, 12, 14 + 23)\\
    \mathfrak{h}_5 = (0, 0, 0, 0, 13 + 42, 14 + 23)\\
    \mathfrak{h}_6 = (0, 0, 0, 0, 12, 13)\\
    \mathfrak{h}_9 = (0, 0, 0, 0,
    12, 14 +
    25).
  \end{array}
\end{equation}
Observe that in all cases
$[\mathfrak{h},\hh]=\textrm{span}\{e_5,e_6\}$. The Lie algebra
$\mathfrak h_5$ corresponds to the Iwasawa manifold, which was
studied in \cite{Di_Scala_2012}.

In order to find the moduli spaces $\mathcal M(H)/{\sim}$, for a
nilpotent simply connected $6$-dimensional Lie group $H$ whose Lie
algebra $\mathfrak{h}$ is one of the Lie algebras listed above, we
will determine in the following sections the corresponding full
automorphism groups. The following lemma picks some common behaviour
present in most of these groups.

\begin{lemma}
  \label{lemaaut}
  Let $\mathfrak h$ be a $2$-step nilpotent Lie algebra of dimension
  $6$ with first Betti number equal to $4$. Let $e_1, \ldots, e_6$ be
  a basis of $\mathfrak h$ such that $[\mathfrak h, \mathfrak h]$ is
  spanned by $e_5, e_6$. Then there exist an algebraic subgroup
  $G \subset \GL_4(\mathbb R)$ and a representation
  $\Delta: G \to \GL_2(\mathbb R)$ such that
  $\Aut(\mathfrak h) \simeq \mathbb R^8 \rtimes G$. More precisely, in
  the above basis, every automorphism of $\mathfrak h$ has the form
  \begin{equation}
    \label{eq:64}
    \begin{pmatrix}
      A & 0 \\
      M & \Delta(A)
    \end{pmatrix}
  \end{equation}
  for some $A \in G$ and $M \in \mathbb R^{2 \times 4} \simeq \mathbb
  R^8$.
\end{lemma}

\begin{proof}
  Clearly every automorphism of $\mathfrak h$ leaves $[\mathfrak h,
  \mathfrak h]$ invariant. The group $G$ is induced by $\Aut(\mathfrak
  h)$ via the projection $\mathfrak h \to \mathfrak h / [\mathfrak h,
  \mathfrak h]$. So we can write any automorphism as \(
  \begin{pmatrix}
    A & 0 \\
    M & B
  \end{pmatrix}
  \). Since
  $[\mathfrak h, \mathfrak h] = \operatorname{span}_{\mathbb R}\{e_5,
  e_6\}$, $B$ depends only on $A$, say $B = \Delta(A)$, and the group
  structure of $\Aut(\mathfrak h)$ forces $\Delta$ to be a
  representation of $G$ in $\mathbb R^2$. Finally, it is easy to see
  that every linear map of the form (\ref{eq:64}) preserves the Lie
  bracket of $\mathfrak h$. Recall that with these identifications,
  $\mathbb R^{2 \times 4} \simeq \mathbb R^8$ is an abelian normal
  subgroup of $\Aut(\mathfrak h)$.
\end{proof}

\section{The case of $\mathfrak h_5 = (0, 0, 0, 0, 13 + 42, 14 + 23)$}

Let $e_1, \ldots, e_6$ be the basis of the Lie algebra $\mathfrak h_5$
such that the only non trivial Lie brackets are
\begin{align*}
  [e_1, e_3] = [e_4, e_2] = -e_5, && [e_1, e_4] = [e_2,e_3] = -e_6.
\end{align*}
It was shown in \cite{Di_Scala_2012} that $\Aut_0(\hh_5)$, the
connected component of the identity of $\Aut(\hh_5)$, is isomorphic to
a twisted (in the sense of Lemma \ref{lemaaut}) semi-direct product
\begin{equation*}
  \mathbb C^{2 \times 2} \rtimes \GL_2(\mathbb{C})
\end{equation*}
and that the the moduli space $\mathcal{M}(\hh_5)$ is homeomorphic to
the product $T \times \Sym_2^+/\sigma$, where $T$ is the triangle
$\{(r, s): 0 < s \le r \le 1\}$ and \( \sigma
\begin{pmatrix}
  E & F \\
  F & G
\end{pmatrix} =
\begin{pmatrix}
  E & -F\\
  -F & G
\end{pmatrix}
\). More precisely, in the standard basis $e_1, \ldots, e_6$, every
left-invariant metric is represented by a unique matrix of the form
\begin{equation}
  \label{eq:69}
  g =
  \begin{pmatrix}
    1 & 0 & 0 & 0 & 0 & 0 \\
    0 & r & 0 & 0 & 0 & 0 \\
    0 & 0 & 1 & 0 & 0 & 0 \\
    0 & 0 & 0 & s & 0 & 0 \\
    0 & 0 & 0 & 0 & E & F \\
    0 & 0 & 0 & 0 & F & G
  \end{pmatrix},
\end{equation}
where $0 < s \le r \le 1$, $EG - F^2 > 0$ and $F \ge 0$.

In this section we will find the whole isometry group of each of the
metrics (\ref{eq:69}).  We start by recalling the following well known
fact that will be used in the sequel.

\begin{remark}
  \label{sec:nilp-lie-algebr}
  Let $H$ be a connected nilpotent Lie group endowed with a
  left-invariant metric $g$. Let us denote by $\mathfrak h$ the Lie
  algebra de $H$. Then by \cite{wolf-1963} (see also
  \cite{Wilson_1982}), the full isometry group of $H$ is given by
  $\I(H, g) = H \rtimes K$, where
  $K = \Aut(\mathfrak h) \cap \OO(\mathfrak h, g)$ under the usual
  identifications.
\end{remark}

\begin{theorem}
  \label{sec:nilp-lie-algebr-1}
  Let $g$ be the left-invariant metric on $H_5$ given in
  (\ref{eq:69}). Then the full isometry group group of $g$ is given by
  \begin{equation*}
    \I(H_5, g) = H_5 \rtimes K
  \end{equation*}
  where $K \simeq \Aut(\mathfrak h) \cap \OO(\mathfrak h, g)$. The
  different subgroups $K$, according to $r, s, E, F, G$ are listed in
  Table \ref{tab:isometric-automorphisms}.
\end{theorem}

\begin{table}[ht]
  \caption{Isotropy subgroups of $\I(H_5, g)$.}
  \centering
  \begin{tabular}[ht]{|l|l|l|}
    \hline
    $K$ & $(r, s)$ & $E, F, G$ \\
    \hline \hline
    $\mathbb Z_2 \times \mathbb Z_2$ & $0 < s < r < 1$ & $F \neq 0$
    \\
    \hline
    $\mathbb Z_2 \times \mathbb Z_2 \times \mathbb Z_2$
        & $0 < s < r < 1$ & $F = 0$ \\
    \hline
    $\mathbb Z_2 \times \mathbb Z_2$ & $0 < s < r = 1$ & $F \neq 0$
    \\
    \hline
    $\mathbb Z_2 \times \mathbb Z_2 \times \mathbb Z_2$
        & $0 < s < r = 1$ & $F = 0, \, G \neq E$ \\
    \hline
    $\OO(2)$ & $0 < s < r = 1$ & $F = 0, \, G = E$ \\
    \hline
    $\OO(2)$ & $0 < s = r < 1$ & $F \neq 0$ \\
    \hline
    $\OO(2) \times \mathbb Z_2$ & $0 < s = r < 1$ & $F = 0$ \\
    \hline
    $\SU(2) \rtimes \mathbb Z_2$ & $s = r = 1$ & $F \neq 0$ \\
    \hline
    $(\SU(2) \rtimes \mathbb Z_2) \rtimes \mathbb Z_2$
        & $s = r = 1$ & $F = 0, \, G \neq E$ \\
    \hline
    $\U(2) \rtimes \mathbb Z_2$ & $s = r = 1$ & $F = 0,\, G = E$ \\
    \hline
  \end{tabular}
  \label{tab:isometric-automorphisms}
\end{table}

It is important to note that the two cases when $K = \OO(2)$ in Table
\ref{tab:isometric-automorphisms} correspond to different subgroups of
$\Aut(\mathfrak h_5)$. These inclusions will become clear in the proof
of Theorem \ref{sec:nilp-lie-algebr-1}.

\begin{proof}
  Let $g$ be given as in (\ref{eq:69}). By using Remark
  \ref{sec:nilp-lie-algebr}, in order to determine the full isometry
  group, we only need to compute the automorphisms of $\mathfrak h_5$
  which are isometric with respect to $g$. Recall that from
  \cite{Di_Scala_2012}, every $\varphi \in \Aut_0(\mathfrak h_5)$ has,
  in the standard basis, the form
  \begin{equation}
    \label{eq:65}
    \varphi =
    \begin{pmatrix}
      A & 0 \\
      M & \Delta(A)
    \end{pmatrix},
  \end{equation}
  where, under de usual identifications,
  $A \in \GL_2(\mathbb C) \subset \GL_4(\mathbb R)$,
  $M \in \mathbb R^{2 \times 4}$ and
  $\Delta(A) = \det A \in \GL_1(\mathbb C) \subset \GL_2(\mathbb
  R)$. Moreover, the full automorphism group of $\mathfrak h_5$ has
  two connected components:
  \begin{equation*}
    \Aut(\mathfrak h_5) = \Aut_0(\mathfrak h_5) \cup \psi
    \Aut_0(\mathfrak h_5),
  \end{equation*}
  where $\psi = \diag(1, -1, 1, -1, 1, -1)$.

  Notice that if $\varphi$ as in (\ref{eq:65}) preserves $g$, then we
  have that $M = 0$, and since $\GL_2(\mathbb C)$ is connected,
  $A \in \SO_{r, s}(4)$ and $\Delta(A) \in \SO_{E, F, G}(2)$, where
  these are the orthogonal groups determined by the $4 \times 4$ and
  $2 \times 2$ nontrivial blocks in $g$. Moreover, we must have
  $\det_{\mathbb R} \Delta(A) = 1$ and so
  $\Delta(A) \in \SO(2) \cap \SO_{E,F,G}(2)$, which implies that
  either $\Delta(A) = \pm I_2$ or $F = 0$ and $G = E$. So, the
  difficult part of the proof is describing
  $\GL_2(\mathbb C) \cap \SO_{r, s}(4)$. Let us write
  \begin{equation}
    \label{eq:68}
    A =
    \begin{pmatrix}
      a_1 & -a_2 & b_1 & -b_2 \\
      a_2 & a_1 & b_2 & b_1 \\
      c_1 & -c_2 & d_1 & -d_2 \\
      c_2 & c_1 & d_2 & d_1
    \end{pmatrix}
    =
    \begin{pmatrix}
      z_1 & z_2 \\
      z_3 & z_4
    \end{pmatrix}
    \in \GL_2(\mathbb C)
  \end{equation}
  and $g' = \diag(1, r, 1, s) = \diag(R, S)$, where $R = \diag(1, r)$
  and $S = \diag(1, s)$. With these identifications, we can write
  the orthogonality condition $A^T g' A = g'$ as
  \begin{equation}
    \label{eq:67}
    \begin{pmatrix}
      \bar z_1 R z_1 + \bar z_3 S z_3 & \bar z_1 R z_2 + \bar z_3 S
      z_4 \\
      \bar z_2 R z_1 + \bar z_4 S z_3 & \bar z_2 R z_2 + \bar z_4 S z_4
    \end{pmatrix}
    =
    \begin{pmatrix}
      R & 0 \\
      0 & S
    \end{pmatrix}
  \end{equation}
  After a close inspection, we notice that the last two entries on the
  diagonal of the left side are $b_1^2 + r b_2^2 + d_1^2 + s d_2^2$
  and $r b_1^2 + b_2^2 + s d_1^2 + d_2^2$. We then equal these values
  to the corresponding entries on the diagonal of $S$ in order to get
  that
  \begin{equation}
    \label{eq:66}
    (r - s) b_1^2 + (1 - rs)b_2^2 + (1 - s^2)d_2^2 = 0.
  \end{equation}
  So we have to study several cases according to the values of $r, s$.

  \textbf{Case $0 < s < r < 1$.} This is the generic case and
  according to (\ref{eq:66}) we have $b_1 = b_2 = d_2 = 0$, which
  forces $d_1 = \pm 1$ and $z_3 = 0$. Therefore, $\bar z_1 R z_1 = R$,
  and as we noticed before, since $r \neq 1$, this implies
  $z_1 = \pm 1$ (i.e.\ $a_1 = \pm 1$ and $a_2 = 0$). Now we check for
  isometric automorphisms in the other connected component. Recall
  that these are all of the form $\psi \varphi$ with
  $\varphi \in \Aut_0(\mathfrak h_5)$. If we keep the notation
  (\ref{eq:65}) and call $\psi' = \diag(1,-1,1,-1)$, then we find that
  $\psi' A$ preserves $g'$. Since $\psi'$ preserves $g'$, we conclude
  that $A$ also preserves $g'$. Thus, if $\psi \varphi$ is an
  isometric automorphism then $F = 0$.

  \textbf{Case $0 < s < r = 1$.} We use (\ref{eq:66}) again in order
  to conclude that $z_2 = z_3 = 0$ and $z_4 = \pm 1$. Since $r = 1$,
  we must have $\bar z_1 z_1 = 1$, which with our identifications
  means that $z_1 \in \SO(2)$. Also, since
  $\Delta(A) = \pm z_1 \in \SO_{E, F, G}(2)$ we see that $F \neq 0$ or
  $G \neq E$ imply $z_1 = \pm 1$. When $F = 0$ and $G = E$, we
  trivially have $\SO_{E,0,E}(2) = \SO(2)$. Finally, with the same
  argument as in the previous case, we can find isometric
  automorphisms outside the connected component of the identity of
  $\Aut(\mathfrak h_5)$ if and only if $F = 0$. Moreover, if in
  addition $G = E$, then $\psi$ is an isometric automorphism which
  lies outside the connected component of the identity of
  $\OO(2) = \OO_{E,0,E}(2)$.

  \textbf{Case $0 < s = r < 1$.} In this case equation \eqref{eq:66}
  becomes $(1 - r^2)(b_2^2 + d_2^2) = 0$, which means
  $z_2, z_4 \in \mathbb R$. With the same idea we used to derive
  \eqref{eq:66}, we can also show that $z_1, z_3 \in \mathbb R$. Now
  looking back to \eqref{eq:67}, with $R = S$ we get that
  \begin{equation*}
    a_1^2 + c_1^2 = b_1^2 + d_1^2 = 1.
  \end{equation*}
  From this, it is not hard to see that the subgroup of
  $\Aut_0(\mathfrak h_4)$ preserving the metric is the intersection of
  $\GL_2(\mathbb R) \subset \GL_2(\mathbb C) \subset \GL_4(\mathbb R)$
  with $\OO(4)$, which is isomorphic to $\OO(2)$. Finally, we will
  have isometric automorphisms other that the ones in
  $\Aut_0(\mathfrak h_5)$ if and only if $\psi$ is isometric, which
  only happens when $F = 0$. Notice that in this case $\psi$ commutes
  with $\OO(2)$, which gives us that the isotropy group of the full
  isometry group is isomorphic to $\OO(2) \times \mathbb Z_2$.

  \textbf{Case $0 < s = r = 1$.} This is the case with most
  symmetries. It is immediate that $A$ as in \eqref{eq:68} belongs to
  $\U(2) = \GL_2(\mathbb C) \cap \OO(4)$. Since $\Delta(A) \in \U(1)$,
  if $F \neq 0$, then $\Delta(A) = \pm 1$ and so
  $A \in \SU(2) \rtimes \mathbb Z_2$. If $F = 0$ and $G \neq E$, we
  also have that $A \in \SU(2) \rtimes \mathbb Z_2$. But $\psi$ is an
  isometric automorphism, then we have two connected component for the
  isometric automorphisms. Finally, if $F = 0$ and $G = E$, then every
  automorphism in $\U(2)$ preserves the metric, and hence the
  isometric automorphisms are isomorphic to
  $\U(2) \rtimes \mathbb Z_2$.
\end{proof}

\section{The case of $\mathfrak h_6 = (0, 0, 0, 0, 12, 13)$}
\label{sec:case-mathfrak-h_6}

\subsection{Automorphism group}
\label{sec:automorphism-group}

Let $\mathfrak h_6$ be the $6$-dimensional $2$-step nilpotent real Lie
algebra corresponding to $(0, 0, 0, 0, 12, 13)$ in Salamon notation
\cite{salamon-2001}. That is, we have a canonical basis
$e_1, \ldots, e_6$ such that the only non-trivial brackets are
$[e_1, e_2] = -e_5$ and $[e_1, e_3] = -e_6$. Equivalently, if
$d: \mathfrak h_6^* \to \Lambda^2(\mathfrak h_6^*)$ is the exterior
derivative on left-invariant forms, then $\ker d$ is spanned by
$e^1, \ldots, e^4$ and $de^5 = e^{12}$, $de^6 = e^{13}$.

It is known from \cite{salamon-2001} that $\mathfrak h_6$ admits an
invariant complex structure. Moreover, according to
\cite{ceballos-2014} there is a unique invariant complex structure up
to equivalence on $\mathfrak h_6$. This means that $\Aut(\mathfrak
h_6)$ acts transitively by conjugation on the set $\mathcal
C(\mathfrak h_6)$ of  invariant complex structures. Recall that the
standard almost complex structure associated to the multiplication
by $\sqrt{-1}$ via the identification $\mathfrak h_6 \simeq \mathbb
R^6 \simeq \mathbb C^3$ is not integrable.

\begin{lemma}
  \label{sec:autom-group-mathfr}
  The invariant almost complex structure $J: \mathfrak h_6 \to
  \mathfrak h_6$ determined by the equations $Je_1 = e_4$, $Je_2 =
  e_3$ and $Je_5 = e_6$ is integrable.
\end{lemma}

\begin{proof}
  Let us denote $\Lambda^{1,0} = \Lambda^{1,0}(\mathfrak
  h_6^*)_{\mathbb C}$ the $i$-eigenspace of $J^*$ on the
  complexification of $\mathfrak h_6$. Notice that $J ^*$ is the
  transpose of $J$, so the equations determining $J^*$ are $J^*e^1 =
  -e^4$, $J^*e^2 = -e^3$ and $J^*e^5 = -e^6$. According to
  \cite{ceballos-2014}, $J$ is integrable if and only if there exists
  a basis $\omega^1, \omega^2, \omega^3$ of $\Lambda^{1, 0}$ such that
  $d\omega^1 = d\omega^2 = 0$ and
  \begin{equation}
    \label{eq:1}
    d\omega^3 = \omega^1 \wedge \omega^2 + \omega^1 \wedge
    \bar\omega^1 + \omega^1 \wedge \bar\omega^2=\omega^1\wedge(\bar\omega^1+2\operatorname{Re}(\omega^2)).
  \end{equation}

  The standard basis of $\Lambda^{1,0}$ associated with the canonical
  basis of $\mathfrak h_6$ is given by
  \begin{align*}
    \eta^1 & = e^1 - iJ^*e^1 = e^1 + ie^4, \\
    \eta^2 & = e^2 - iJ^*e^2 = e^2 + ie^3, \\
    \eta^3 & = e^5 - iJ^*e^5 = e^5 + ie^6.
  \end{align*}

  Suppose that there exist $\omega^1, \omega^2, \omega^3$ as in
  \eqref{eq:1}. We can assume that $\omega^3 = \eta^3$, and so
  $d\omega^3 = e^{12} + ie^{13}$. We can further assume that
  $\omega^1, \omega^2$ belong to the subspace spanned by
  $\eta^1, \eta^2$. If we write $\omega^1 = A\eta^1 + B\eta^2$ then,
  taking the imaginary part of both sides of equation \eqref{eq:1},
  we get that that $A = 0$ and $B \neq 0$. One can also assume that
  $B = 1$ and so, $\omega^1 = \eta^2$. Let us write
  $\omega^2 = C\eta^1 + D\eta^2$. Then replacing it in \eqref{eq:1}
  one gets
  \begin{align*}
    e^{12} + ie^{13} & = (e^2 + ie^3) \wedge (2Ce^1 + (2D+ 1)e^2 -
                       ie^3) \\
                     & = -2C(e^{12} + ie^{13}) - 2i(D + 1)e^{23}
  \end{align*}
  So $C = -\frac12$, $D = -1$ and $\omega^1 = \eta^2$,
  $\omega^2 = -\frac12\eta^1 - \eta^2$, $\omega^3 = \eta^3$ is the
  basis of $\Lambda^{1,0}$ we were looking for.
\end{proof}

\begin{lemma}
  \label{sec:autom-group-mathfr-1}
  If $f \in \Aut(\mathfrak h_6)$ then:
  \begin{enumerate}
  \item \label{item:1} $e^1(f(e_j)) = 0$ for $j = 2, \ldots, 6$;
  \item \label{item:2} $e^2(f(e_j)) = 0$ for $j = 4, 5, 6$;
  \item \label{item:3} $e^3(f(e_j)) = 0$ for $j = 4, 5, 6$;
  \item \label{item:4} $e^4(f(e_j)) = 0$ for $j = 5, 6$;
  \item \label{item:5} $e^5(f(e_j)) = e^1(f(e_1))e^2(f(e_{j -3}))$,
    for $j = 5, 6$;
  \item \label{item:6} $e^6(f(e_j)) = e^1(f(e_1))e^3(f(e_{j - 3})),$
    for $j = 5, 6$.
  \end{enumerate}
\end{lemma}

\begin{proof}
  Since the center of $\mathfrak h_6$ is spanned by $e_4, e_5, e_6$
  and $f$ leaves the center invariant, we get that $e^k(f(e_j)) = 0$
  for all $k = 1, 2, 3$ and $j = 4, 5, 6$.  Also, since
  $\dim(\ker \ad_{e_j})$ is preserved under automorphisms,
  $e^1(f(e_j)) = 0$ if $j \ge 2$. These two observations together
  prove (\ref{item:1}), (\ref{item:2}) and (\ref{item:3}). Part
  (\ref{item:4}) follows from Lemma \ref{lemaaut}. For parts
  (\ref{item:5}) and (\ref{item:6}) recall that $e_5 = -[e_1,
  e_2]$. Then
  \begin{align*}
    f(e_5) & = -[f(e_1), f(e_2)] \\
           & = -[e^1(f(e_1))e_1, e^2(f(e_2))e_2 + e^3(f(e_2))e_3] \\
           & = e^1(f(e_1))e^2(f(e_2))e_5 + e^1(f(e_1))e^3(f(e_2))e_6.
  \end{align*}
  With the same argument we can see that
  \begin{equation*}
    f(e_6) = e^1(f(e_1))e^2(f(e_3))e_5 + e^1(f(e_1))e^3(f(e_3))e_6. \qedhere
  \end{equation*}

\end{proof}

\begin{lemma}
  \label{sec:autom-group-mathfr-2}
  Let $J$ be the complex structure of Lemma
  \ref{sec:autom-group-mathfr}. Then the isotropy subgroup at $J$ of
  $\Aut(\mathfrak h_6)$ is isomorphic to
  \begin{equation*}
    \mathbb R^4 \rtimes_\varphi (\GL_1(\mathbb R) \times
    \GL_1(\mathbb C)),
  \end{equation*}
  where $\varphi: \GL_1(\mathbb R) \times \GL_1(\mathbb C) \to
  \GL_4(\mathbb R)$ is the representation given by
  \begin{equation*}
    \varphi(r, a + ib) =
    \begin{pmatrix}
      a & 0 & b & 0 \\
      0 & r & 0 & 0 \\
      -b & 0 & a & 0 \\
      0 & 0 & 0 & r
    \end{pmatrix},
  \end{equation*}
  for $r \neq 0$ and $a^2 + b^2 \neq 0$.
\end{lemma}

\begin{proof}
  Let $f \in \Aut(\mathfrak h_6)$ and identify it with its matrix
  $(a_{ij})$ in the basis $e_1, \ldots, e_6$. From Lemma
  \ref{sec:autom-group-mathfr-1} we must have
  \begin{equation}
    \label{eq:2}
  \begin{aligned}
    & a_{1j} =0 \text{ for } j \geq 2,
    && a_{2j} = a_{3j} =0 \text{ for } j\geq 4,
    && a_{45} = a_{46} = 0,
    && \\
    & a_{55} = a_{11} a_{22},
    && a_{56} = a_{11} a_{23},
    && a_{65} = a_{11} a_{32},
    && a_{66}= a_{11}a_{33}.
  \end{aligned}
\end{equation}
  If, in addition, we ask $f$ to be in the isotropy of $J$, then $f$
  must commute with the matrix associated to $J$, and thus it has the
  form
  \begin{equation}
    \label{eq:3}
    f =
    \begin{pmatrix}
      a_{11} & 0 & 0 & 0 & 0 & 0 \\
      0 & a_{22} & a_{23} & 0 & 0 & 0 \\
      0 & -a_{23} & a_{22} & 0 & 0 & 0 \\
      0 & 0 & 0 & a_{11} & 0 & 0 \\
      a_{51} & a_{52} & -a_{62} & -a_{61} & a_{11}a_{22} & a_{11}a_{23} \\
      a_{61} & a_{62} & a_{52} & a_{51} & -a_{11}a_{23} & a_{11}a_{22}
    \end{pmatrix}.
  \end{equation}
  with $a_{11} \neq 0$ and $a_{22}^2 + a_{23}^2 \neq 0$. Moreover,
  every linear map of the form \eqref{eq:3} is an automorphism of
  $\mathfrak h_6$. In order to see this, we can show that
  $f^* \circ d = d \circ f^*$. Since $\ker d$ is spanned by
  $e^1, e^2, e^3, e^4$, clearly $f^*(de^j) = d(f^*(e^j))$ for
  $1 \le j \le 4$. Also,
  \begin{align*}
    d(f^*(e^5)) & = a_{11}a_{22}e^{12} + a_{11}a_{23}e^{13} \\
                & = a_{11}e^1 \wedge (a_{22}e^2 + a_{23}e^3) \\
                & = f^*(e^1) \wedge f^*(e^2) \\
                & = f^*(e^{12}) = f^*(de^5)
  \end{align*}
  and similarly $d(f^*(e^6)) = f^*(de^6)$. Hence $f$ commutes with $J$
  if and only it has the form~\eqref{eq:3}.

  Finally, notice that $\Aut(\mathfrak h_6)_J = K \rtimes H$ is the
  inner semi-direct product of the normal subgroup
  \begin{equation*}
    K = \left\{
      \begin{pmatrix}
        1 & 0 & 0 & 0 & 0 & 0 \\
        0 & 1 & 0 & 0 & 0 & 0 \\
        0 & 0 & 1 & 0 & 0 & 0 \\
        0 & 0 & 0 & 1 & 0 & 0 \\
        a_{51} & a_{52} & -a_{62} & -a_{61} & 1 & 0 \\
        a_{61} & a_{62} & a_{52} & a_{51} & 0 & 1
      \end{pmatrix}: a_{51}, a_{52}, a_{61}, a_{62} \in \mathbb R
    \right\} \simeq \mathbb R^4
  \end{equation*}
  and the subgroup
  \begin{align*}
    H & = \left\{
        \begin{pmatrix}
          a_{11} & 0 & 0 & 0 & 0 & 0 \\
          0 & a_{22} & a_{23} & 0 & 0 & 0 \\
          0 & -a_{23} & a_{22} & 0 & 0 & 0 \\
          0 & 0 & 0 & a_{11} & 0 & 0 \\
          0 & 0 & 0 & 0 & a_{11}a_{22} & a_{11}a_{23} \\
          0 & 0 & 0 & 0 & -a_{11}a_{23} & a_{11}a_{22}
        \end{pmatrix}: a_{11} \neq 0, a_{22}^2 + a_{23}^2 \neq 0
                                          \right\} \\
      & \simeq \GL_1(\mathbb R) \times \GL_1(\mathbb C).
  \end{align*}
  Now one can easily check that
  $\Aut(\mathfrak h_6)_{J} \simeq \mathbb R^4
  \rtimes_{\varphi}(\GL_1(\mathbb R) \times \GL_1(\mathbb C))$ as
  stated.
\end{proof}

It is convenient to introduce some notation before stating the main
result of this section. Let us consider the presentation of the
$5$-dimensional Heisenberg Lie group
$\Heis_2 = \{(x, y, z): x, y \in \mathbb R^2,\, z \in \mathbb R\}$
with the multiplication given by
\begin{equation*}
  (x, y, z)(x', y', z') = (x + x', y + y', z + z' + y \cdot x')
\end{equation*}
and let $G$ be the subgroup of $\GL_4(\mathbb R)$ consisting of all
the matrices in block form
\begin{equation}
  \label{eq:5}
  A =
  \begin{pmatrix}
    r & & \\
    x & \tilde A & \\
    z & y^T & s
  \end{pmatrix}
\end{equation}
where $r, s \in \mathbb R - \{0\}$, $\tilde A \in \GL_2(\mathbb R)$, $x, y
\in \mathbb R^2 \simeq \mathbb R^{2 \times 1}$ and $z \in \mathbb
R$. It is not hard to see that
\begin{equation}
  \label{eq:4}
  G \simeq (\Heis_2 \rtimes_{\varphi_1} \GL_2(\mathbb R))
  \rtimes_{\varphi_2} (\GL_1(\mathbb R) \times \GL_1(\mathbb R))
\end{equation}
where $\varphi_1: \GL_2(\mathbb R) \to \Aut(\Heis_2)$ and
$\varphi_2: \GL_1(\mathbb R) \times \GL_1(\mathbb R) \to \Aut(\Heis_2
\rtimes_{\varphi_1} \GL_2(\mathbb R))$ are the Lie groups morphisms
given by
\begin{align*}
  \varphi_1(\tilde A)(x, y, z) & = (\tilde Ax, (\tilde A^{-1})^Ty, z)
  \\
  \varphi_2(r, s)(x, y, z, \tilde A) & = \left(\frac x r, s(\tilde
                                       A^{-1})^Ty, \frac {s z} r, \tilde
                                       A\right)
\end{align*}

Let us also consider the Lie groups epimorphism $\Delta: G \to
\GL_2(\mathbb R)$ defined by
\begin{equation}
  \label{eq:6}
  \Delta(A) = r\tilde A.
\end{equation}
Recall that
after the identification of $G$ given by (\ref{eq:4}), the kernel of
$\Delta$ is a normal subgroup isomorphic to $\Heis_2 \rtimes
\GL_1(\mathbb R)$.

\begin{theorem}
  \label{sec:autom-group-mathfr-3}
  Let $G$ be the Lie subgroup of $\GL_4(\mathbb R)$ defined in
  (\ref{eq:5}). There exists an isomorphism of Lie groups
  \begin{equation*}
    \Aut(\mathfrak h_6) \simeq \mathbb R^{2 \times 4}
    \rtimes_{\varphi} G,
  \end{equation*}
  where $\mathbb R^{2 \times 4}$ is the abelian Lie group of $2 \times
  4$ matrices and $\varphi: G \to \GL(\mathbb R^{2 \times 4})$ is given
  by $\varphi(A)M = \Delta(A)MA^{-1}$, being $\Delta$ defined as in
  (\ref{eq:6}). Moreover, every automorphism of $\mathfrak h_6$ is
  represented in the canonical basis by a matrix of the form
  \begin{equation}
    \label{eq:7}
    \begin{pmatrix}
      A & 0 \\
      M & \Delta(A)
    \end{pmatrix}
  \end{equation}
  where $A \in G$ and $M \in \mathbb R^{2 \times 4}$.
\end{theorem}

\begin{proof}
  Let $\tilde G$ the subgroup consisting of all the matrices of the
  form (\ref{eq:7}). Recall that this subgroup agrees with the one
  defined by the equations (\ref{eq:2}). So, from Lemma
  \ref{sec:autom-group-mathfr-1} and the above paragraphs,
  \begin{equation*}
    \Aut(\mathfrak h_6) \subset \tilde G \simeq \mathbb R^{2 \times 4}
    \rtimes_{\varphi} G.
  \end{equation*}

  Now, it follows from \cite{salamon-2001} that
  $\mathcal C(\mathfrak h_6)$ has real dimension $12$. Since
  $\Aut(\mathfrak h_6)$ is transitive on $\mathcal C(\mathfrak h_6)$,
  it follows from Lemma \ref{sec:autom-group-mathfr-2} that
  $\dim \Aut(\mathfrak h_6) = \dim \tilde G = 19$. So, the identity
  components of $\Aut(\mathfrak h_6)$ and $\tilde G$ coincide. In
  order to see $\Aut(\mathfrak h_6) = \tilde G$, it is enough to see
  that there is an automorphism of $\mathfrak h_6$ in each of the
  other seven connected components of $\tilde G$. Let us choose the
  following representatives for the connected components of
  $\tilde G$:
  \begin{align*}
    f_1 & = I_6 & f_5 & = \diag(-1,1,1,1,-1,-1), \\
    f_2 & = \diag(1,1,1,-1,1,1), & f_6 & = \diag(-1,1,1,-1,-1,-1), \\
    f_3 & = \diag(1,1,-1,1,1,-1), & f_7 & = \diag(-1,1,-1,1,-1,1), \\
    f_4 & = \diag(1,1,-1,-1,1,-1), & f_8 & = \diag(-1,1,-1,-1,-1,1). \\
  \end{align*}

  Since the $f_j$'s form a subgroup of $\tilde G$ and every $f_j$ but
  $f_1$ has order $2$, it is enough to show that three out of $f_2,
  \ldots, f_8$ are in $\Aut(\mathfrak h_6)$. Moreover, from Lemma
  \ref{sec:autom-group-mathfr-2}, $f_6 \in \Aut(\mathfrak h_6)$. Let
  us verify that also $f_2, f_3 \in \Aut(\mathfrak h_6)$. Reasoning as
  in the proof of Lemma \ref{sec:autom-group-mathfr-2},
  \begin{align*}
    d(f_2^*(e^5)) & = de^5 = e^{12} = f_2^*(e^{12}) = f_2^*(de^5)  \\
    d(f_2^*(e^6)) & = de^6 = e^{13} = f_2^*(e^{13}) = f_2^*(de^6) \\
    d(f_3^*(e^5)) & = de^5 = e^{12} = f_3^*(e^{12}) = f_3^*(de^5)  \\
    d(f_3^*(e^6)) & = -de^6 = -e^{13} = f_3^*(e^{13}) = f_3^*(de^6)
  \end{align*}

  This concludes the proof of the theorem.
\end{proof}

The following result is an immediate consequence of the proof of
Theorem \ref{sec:autom-group-mathfr-3}.

\begin{corollary}
  $\Aut(\mathfrak h_6)$ has $8$ connected components. Moreover,
  \begin{equation*}
    \Aut(\mathfrak h_6)/\Aut_0(\mathfrak h_6) \simeq \mathbb Z_2
    \oplus \mathbb Z_2 \oplus \mathbb Z_2.
  \end{equation*}
\end{corollary}

\subsection{Left-invariant metrics}

Consider an inner product $g$ on $\mathfrak h_6$. Then in the canonical
basis, $g$ can be represented by a symmetric positive definite matrix
of the form
\begin{equation}
  \label{metrica}
  g = \left(
    \begin{array}{cc}
      B & C^T\\
      C & D
    \end{array}
  \right)
\end{equation}
with $B \in \Sym_4^+$, $D \in \Sym_2^+$ and
$C \in \mathbb R^{2\times 4}$.

\begin{lemma}
  \label{sec:left-invar-metr-6}
  Let $G$ be the subgroup of $\GL_4(\mathbb R)$ defined in
  (\ref{eq:5}). Then:
  \begin{enumerate}
  \item \label{item:7} $G$ acts transitively on $\Sym_4^+$.
  \item \label{item:8} Any metric $g$ on $\mathfrak h_6$ is
    equivalent, by an automorphism in
    $G \subset \Aut(\mathfrak{h_6})$, to a metric of the form
    \(
    \begin{pmatrix}
      I_4 & \tilde{C}^T \\
      \tilde{C} & \tilde{D}
    \end{pmatrix}
    \).
  \item \label{item:9} Any metric $g$ on $\mathfrak h_6$ is
    equivalent, by an automorphism in
    $\mathbb{R}^{8}\subset \Aut(\mathfrak h_6)$ to a metric of the
    form
    \(
    \begin{pmatrix}
      \tilde{A} & 0 \\
      0 & \tilde{D}
    \end{pmatrix}
    \).
  \end{enumerate}
\end{lemma}

Recall that the inclusions $G \subset \Aut(\mathfrak h_6)$ and
$\mathbb R^8 \simeq \mathbb R^{2 \times 4} \subset \Aut(\mathfrak
h_6)$ are the ones provided by Theorem \ref{sec:autom-group-mathfr-3}.

\begin{proof}
  Observe that any element of $\Sym_4^+$ can be written as $X^{T}X$,
  where $X$ is a lower-triangular matrix. Since the set of
  lower-triangular matrices is contained in $G$, we conclude that any
  element of $\Sym_4^{+}$ is in the orbit of the identity. This proves
  part (\ref{item:7}). Item~(\ref{item:8}) follows from the first
  property. In fact, if $g$ has the form given in
  equation~(\ref{metrica}), one only needs to choose an element
  $A \in G \subset \Aut(\mathfrak h_6)$ such that $A^TBA =
  I_4$. Finally, for (\ref{item:9}), let
  $M \in \mathbb R^{2\times 4} \simeq \mathbb R^8 \subset
  \Aut(\mathfrak h_6)$. Then if $g$ is as in equation (\ref{metrica}),
  \begin{equation*}
    M^T g M =
    \begin{pmatrix}
      \tilde B & \tilde C^T \\
      \tilde C & D
    \end{pmatrix}
  \end{equation*}
  where $\tilde B = B + C^TM + M^TC + M^TDM$ and $\tilde C = C +
  DM$. Choosing $M = -D^{-1}C$ we get the desired result.
\end{proof}

\begin{corollary}\label{coro}
  Any inner product $g$ on $\mathfrak h_6$ is equivalent via an
  element of $\Aut(\mathfrak h_6)$ to one of the form
  \(
  \begin{pmatrix}
    I_4 & 0 \\
    0 & \tilde{D}
  \end{pmatrix},
  \)
  with $\tilde{D}\in \Sym_2^+$.
\end{corollary}

\begin{remark}\label{rem}
  Observe that any two inner products $g$, $g'$ given by matrices
  $\tilde{D}$, $\tilde{D}'$ as in Corollary~\ref{coro} are equivalent
  by an automorphism of $\mathfrak h_6$ if and only if the matrices
  $\tilde{D}$ and $\tilde{D'}$ are conjugated by an element of
  $\OO(2) \subset G \subset \Aut(\mathfrak h_6)$. Since
  $\Sym_2^+ = \GL_2(\mathbb R)/\OO(2)$, each family of equivalent
  metrics can be identified with an orbit of the isotropy action in
  this symmetric space.
\end{remark}

\begin{theorem}
  \label{sec:left-invar-metr-2}
  Let $H_6$ be the simply connected Lie group with Lie algebra
  $\mathfrak h_6$. Each left-invariant metric on $H_6$ is equivalent
  by an automorphism to a metric of the form
  \begin{equation}
    \label{eq:17}
    g = \sum_{i = 1}^4 e^i \otimes e^i + a e^5 \otimes e^5 + b e^6 \otimes e^6,
  \end{equation}
  with $a,b>0$. Moreover, the moduli space $\mathcal M(H_6)/{\sim}$ is
  homeomorphic to
  $$
  \{(a,b) \in \mathbb R^2: \,0<a\leq b\}.
  $$
\end{theorem}

\begin{proof}
  By Remark \ref{rem}, we only need to find a section to the orbits of
  the $\OO(2)$-action on the symmetric space
  $\Sym_2^+ = \GL_2(\mathbb R)/\OO(2)$.  Observe that
  $\mathfrak{gl}_2(\mathbb R) = \so(2) + \Sym_2$ is a Cartan
  decomposition of $\mathfrak{gl}_2(\mathbb R)$, where $\Sym_2$ denote
  the subspace of symmetric matrices. So, a section of the
  $\OO(2)$-action on $\Sym_2^+$ is the exponential of a maximal
  abelian subalgebra of $\Sym_2$, which is given by the $2 \times 2$
  diagonal matrices. This proves the first assertion. The second one
  is a consequence of the fact that conjugation by $J_0 \in \OO(2)$ of
  a diagonal matrix interchanges the diagonal entries, where $J_0$
  denotes the multiplication by $\sqrt{-1}$ in
  $\mathbb C \simeq \mathbb R^2$.
\end{proof}

\begin{corollary}
  \label{sec:left-invar-metr-1}
  Let $g_{a, b}$ be the left-invariant metric on $H_6$ given by
  \eqref{eq:17}. Then the full isometry group of $g_{a, b}$ is given by
  \begin{equation*}
    \I(H_6, g_{a, b}) =
    \begin{cases}
      H_6 \rtimes (\OO(2) \times \mathbb Z_2 \times \mathbb Z_2), & a
      = b \\
      H_6 \rtimes (\mathbb Z_2 \times \mathbb Z_2 \times \mathbb Z_2)
      & a \neq b
    \end{cases}
  \end{equation*}
\end{corollary}

\begin{proof}
  According to Remark \ref{sec:nilp-lie-algebr}, we only need to
  compute the isometric automorphisms of $\mathfrak h_6$.  From
  Theorem \ref{sec:autom-group-mathfr-6}, an automorphism $f$ of
  $\mathfrak h_6$ has the form
  \begin{equation*}
    \begin{pmatrix}
      r & & & \\
      x & A & & \\
      z & y^T & s \\
      M_1 & M_2 & M_3 & rA
    \end{pmatrix}
  \end{equation*}
  in the canonical basis, where $r, s \in \mathbb R - \{0\}$,
  $z \in \mathbb R$, $x, y, M_1, M_3 \in \mathbb R^{2 \times 1}$,
  $A \in \GL_2(\mathbb R)$ and $M_2 \in \mathbb R^{2 \times 2}$. If
  $f$ leaves $g_{a, b}$ invariant, then $x, y, z, M_1, M_2, M_2$ all
  vanish, $r, s \in \{\pm 1\}$ and $A \in \OO(2)$. Moreover, if
  $a \neq b$ then $A = \pm I_2$. This implies the result.
\end{proof}

\section{The case of $\mathfrak h_4 = (0, 0, 0, 0, 12, 14 + 23)$}
\label{sec:case-mathfrak-h_4}

\subsection{Automorphism group}

Let us consider the basis of $\mathfrak h_4$ whose only non vanishing
are differentials are $de^5 = e^{12}$ and $de^6 = e^{14} + e^{23}$. In
terms of the Lie bracket, we can assume that the only non trivial
brackets in the above basis are
\begin{align*}
  [e_1, e_2] = -e_5, && [e_1, e_4] = [e_2, e_3] = -e_6.
\end{align*}

\begin{lemma}
  \label{sec:autom-group-mathfr-4}
  Let $f \in \Aut(\mathfrak h_4)$, then:
  \begin{enumerate}
  \item $e^i(f(e_j)) = 0$ for $i = 1,2$ and $j = 3, 4$;
  \item $e^i(f(e_5)) = 0$ for $i = 1, \ldots, 4$;
  \item $e^i(f(e_6)) = 0$ for $i = 1, \ldots, 5$.
  \end{enumerate}
\end{lemma}

\begin{proof}
  Since $f$ is an automorphism, it leaves invariant $\dim(\ker \ad_x)$
  for all $x \in \mathfrak h_4$. In particular, if $j = 3, 4$, then
  $\dim(\ker \ad_{f(e_j)}) = 1$ and hence
  $e^1(f(e_j)) = e^2(f(e_j)) = 0$. Now if $j = 5, 6$ then
  $e^i(f(e_j)) = 0$ for $i = 1, \ldots 4$, since $f$ leaves the center
  of $\mathfrak h_4$ invariant. Moreover, since
  $f(e_6) = -[f(e_1), f(e_4)]$ and the $e_1$- and $e_2$-components of
  $f(e_4)$ are zero, it follows that $e^{5}(f(e_6)) = 0$.
\end{proof}

In order to compute the full automorphism group of $\mathfrak h_4$, it
is easier to determine first the connected component of the
identity. Recall that the Lie algebra of $\Aut(\mathfrak h_4)$ is
given by the derivations of $\mathfrak h_4$,
\begin{equation*}
  \Der(\mathfrak h_4) = \{D \in \mathfrak{gl}(\mathfrak h_4):
  D[X,Y] = [DX,Y] + [X, DY] \text{ for all } X, Y \in \mathfrak
  h_4\}.
\end{equation*}
Identifying, as usual, $D$ with its matrix in the basis
$e_1, \ldots, e_6$, the conditions
$D[e_i, e_j] = [De_i, e_j] + [e_i, De_j]$, for $i < j$, define a
linear system in the entries of $D$. A straight-forward computation,
together with Lemma \ref{sec:autom-group-mathfr-4}, allows us to prove
the following fact.

\begin{lemma}
  \label{sec:autom-group-mathfr-5}
  The Lie algebra $\Der(\mathfrak h_4)$, after the usual
  identification, is given by the Lie subalgebra of
  $\mathfrak{gl}_6(\mathbb R)$ which consists of the matrices of the
  following form
  \begin{equation}\label{eq:8}
    D =
    \begin{pmatrix}
      d_{11} & d_{12} & 0 & 0 & 0 & 0 \\
      d_{21} & d_{22} & 0 & 0 & 0 & 0 \\
      d_{31} & d_{32} & d_{11} + x & -d_{12} & 0 & 0 \\
      d_{41} & d_{42} & -d_{21} & d_{22} + x & 0 & 0 \\
      d_{51} & d_{52} & d_{53} & d_{54} & d_{11} + d_{22} & 0 \\
      d_{61} & d_{62} & d_{63} & d_{64} & -d_{31} + d_{42} & d_{11} +
      d_{22} + x \\
    \end{pmatrix},
  \end{equation}
  where $x, d_{ij} \in \mathbb R$.  In particular,
  $\dim \Aut(\mathfrak h_4) = 17$.
\end{lemma}

In order to describe the full automorphism group, we introduce the
following notation. Let
$\sigma: \GL_2(\mathbb R) \to \GL_2(\mathbb R)$ be the Lie involution
given by
\begin{equation*}
  \sigma
  \begin{pmatrix}
    a & b \\
    c & d
  \end{pmatrix}
  =
  \begin{pmatrix}
    a & -b \\
    -c & d
  \end{pmatrix}
\end{equation*}
and let $(\cdot,\cdot)$ the semi-definite inner product on
$\mathfrak{gl}_2(\mathbb R)$ defined as
\begin{equation}
  \label{eq:13}
  (A, B) = a_{11}b_{22} - a_{12}b_{21} + a_{22}b_{12} - a_{22}b_{11}.
\end{equation}

Consider the closed Lie subgroup $G \subset \GL_4(\mathbb R)$ of
matrices of the form
\begin{equation}\label{eq:9}
  \begin{pmatrix}
    A & 0 \\
    B & x\sigma(A)
  \end{pmatrix}
\end{equation}
where $A \in \GL_2(\mathbb R)$, $B \in \mathbb R^{2 \times 2}$ and $x
\in \mathbb R - \{0\}$. It follows that $G$ is isomorphic to the
semi-direct product
\begin{equation*}
  G \simeq \mathbb R^{2 \times 2} \rtimes_{\varphi_1} (\GL_2(\mathbb R)
  \times \GL_1(\mathbb R))
\end{equation*}
where
\begin{equation*}
  \varphi_1(A, x) B = x\sigma(A)BA^{-1}.
\end{equation*}

Finally,  consider the representation $\Delta: G \to \GL_2(\mathbb R)$ given by
\begin{equation}
  \label{eq:12}
  \Delta
  \begin{pmatrix}
    A & 0 \\
    B & x\sigma(A)
  \end{pmatrix}
  =
  \begin{pmatrix}
    \det A & 0 \\
    (A, B) & x \det A
  \end{pmatrix}.
\end{equation}

\begin{theorem}
  \label{sec:autom-group-mathfr-6}
  Let $G$ be the Lie subgroup of $\GL_4(\mathbb R)$ defined in
  (\ref{eq:9}). There exists an isomorphism of Lie groups
  \begin{equation*}
    \Aut(\mathfrak h_4) \simeq \mathbb R^{2 \times 4} \rtimes_\varphi G,
  \end{equation*}
  where $\mathbb R^{2 \times 4}$ is the abelian Lie group of
  $2 \times 4$ real matrices and
  $\varphi: G \to \GL(\mathbb R^{2 \times 4})$ is the representation
  given by $\varphi(A)M = \Delta(A)MA^{-1}$, with $\Delta$ defined as
  in (\ref{eq:12}). Moreover, any automorphism of $\mathfrak h_4$ is
  represented in the canonical basis by a matrix of the form
  \begin{equation}
    \label{eq:10}
    \begin{pmatrix}
      A & 0 \\
      M & \Delta(A)
    \end{pmatrix},
  \end{equation}
  where $A \in G$ and $M \in \mathbb R^{2 \times 4}$.
\end{theorem}

\begin{proof}
  Let $\tilde G$ the Lie subgroup of $\GL_6(\mathbb R)$ which consists
  of all the matrices of the form~\eqref{eq:10}. It follows from Lemma
  \ref{sec:autom-group-mathfr-5} that
  $\tilde G \subset \Aut(\mathfrak h_4)$. Moreover, since these two
  groups have dimension $17$, their connected components coincide. It
  only remains to show that $\Aut(\mathfrak h_4)$ has no other
  connected components apart from the ones given by $\tilde G$. From
  Lemma \ref{sec:autom-group-mathfr-4}, we know that any
  $f \in \Aut(\mathfrak h_4)$, with matrix $(a_{ij})$ in the basis
  $e_1, \ldots, e_6$, is such that
  \begin{align}
    \label{eq:11}
    a_{1j} = a_{2j} = 0 \text{ for } j\geq 3,
    &&
       a_{3j} = a_{4j} = 0 \text{ for } j \geq 5, && a_{56}=0.
  \end{align}
  Of course, some of these parameters are dependent on the
  others. Lets call
  \begin{align*}
    A =
    \begin{pmatrix}
      a_{11} & a_{12} \\
      a_{21} & a_{22}
    \end{pmatrix},
             &&
                B =
                \begin{pmatrix}
                  a_{31} & a_{32} \\
                  a_{41} & a_{42}
                \end{pmatrix},
             &&
                C =
                \begin{pmatrix}
                  a_{33} & a_{34} \\
                  a_{43} & a_{44}
                \end{pmatrix}.
  \end{align*}

  Since $[f(e_1), f(e_2)] = -f(e_5)$, we easily check that
  $a_{55} = \det A$ and $a_{65} = (A, B)$ where $(\cdot, \cdot)$ is
  the bilinear form defined in (\ref{eq:13}). Using that
  $[f(e_1), f(e_4)] = [f(e_2), f(e_3)] = -f(e_6)$ and
  $[f(e_1), f(e_3)] = [f(e_2), f(e_4)] = 0$, we obtain the following
  equations:
  \begin{align*}
    a_{11}a_{43} + a_{21}a_{33}  = a_{22}a_{34} + a_{12}a_{44} & = 0  \\
    a_{11}a_{44} + a_{21}a_{34} = a_{22}a_{33} + a_{12}a_{43} & = -a_{66}
  \end{align*}

  We can rewrite the above system as
  \begin{equation*}
    \operatorname{adj}(\sigma(C)) A = -a_{66}I_2
  \end{equation*}
  and hence the only possible solution is
  \begin{align*}
    C = x \sigma(A), && a_{66} = x \det(A)
  \end{align*}
  for some $x \neq 0$, as we wanted to show.
\end{proof}

\begin{corollary}
  $\Aut(\mathfrak h_4)$ has $4$ connected components. Moreover,
  \begin{equation*}
    \Aut(\mathfrak h_4)/\Aut_0(\mathfrak h_4) \simeq \mathbb Z_2
    \oplus \mathbb Z_2.
  \end{equation*}

\end{corollary}

Recall that $\mathcal C(\mathfrak h_4)$ is an algebraic variety and
according to \cite{salamon-2001},
$\dim \mathcal C( \mathfrak h_4) = 12$.

\begin{corollary}
  \label{sec:automorphism-group-3}
  $\mathcal A(\mathfrak h_4)$ is a $9$-dimensional smooth manifold.
\end{corollary}

\begin{proof}
  According to \cite{Andrada_2011}, $\Aut(\mathfrak h_4)$ is
  transitive on $\mathcal C(\mathfrak h_4)$. Moreover, every abelian
  structure on $\mathfrak h_4$ is conjugated by an automorphism to the
  one given, in the canonical basis, by the matrix
  \begin{equation*}
    J =
    \begin{pmatrix}
      0 & -1 & 0 & 0 & 0 & 0 \\
      1 & 0 & 0 & 0 & 0 & 0 \\
      0 & 0 & 0 & 1 & 0 & 0 \\
      0 & 0 & -1 & 0 & 0 & 0 \\
      0 & 0 & 0 & 0 & 0 & -1 \\
      0 & 0 & 0 & 0 & 1 & 0
    \end{pmatrix}.
  \end{equation*}
  Let $f$ be an automorphism of $\mathfrak h_4$. Assume that $f$ is
  represented in the canonical basis by the matrix
  \begin{equation*}
    \begin{pmatrix}
      A & 0 & 0 \\
      B & x \sigma(A) & 0 \\
      M_1 & M_2 & \Delta(A)
    \end{pmatrix}
  \end{equation*}
  with $A \in \GL_2(\mathbb R)$,
  $B, M_1, M_2 \in \mathbb R^{2 \times 2}$ and $x \neq 0$. Notice that
  we are keeping the notation of Theorem
  \ref{sec:autom-group-mathfr-6} but replacing the matrix $M$ by the
  square matrices $M_1$ and $M_2$. It follows that $f$ commutes with
  $J$ if and only if, after the usual identifications, $x = 1$,
  $A \in \GL_1(\mathbb C) \subset \GL_2(\mathbb R)$, and
  $B, M_1, M_2 \in \mathbb C \subset \mathbb R^{2 \times 2}$. In
  particular,
  \begin{equation*}
    \mathcal A(\mathfrak h_4) \simeq \frac{\mathbb R^8 \rtimes (\mathbb
      R^4 \rtimes (\GL_2(\mathbb R) \times \GL_1(\mathbb R)))}{\mathbb
      C^2 \rtimes (\mathbb C \rtimes \GL_1(\mathbb C))}
  \end{equation*}
  and the corollary follows.
\end{proof}

\subsection{Left-invariant metrics}

We have seen that
$\Aut(\mathfrak h_4) = \mathbb R^{2 \times 4} \rtimes G$, where $G$ is
the subgroup of $\GL_4(\mathbb R)$ defined in (\ref{eq:9}). We shall
study first the action of $G$ on the symmetric space
$\Sym_4^+ = \GL_4(\mathbb R)/{\OO(4)}$. Recall, as usual, that a generic
element $g \in \Sym_4^+$ has the form
\begin{equation*}
  \begin{pmatrix}
    P & Q \\
    Q^T & R
  \end{pmatrix}
\end{equation*}
where $P, R \in \Sym_2^+$. On the other hand, any element
$\varphi \in G$ has the form
\begin{equation}
  \label{eq:14}
  \begin{pmatrix}
    A & 0 \\
    B & x\sigma(A)
  \end{pmatrix}
\end{equation}
where $A \in \GL_2(\mathbb R)$, $B \in \mathbb R^{2 \times 2}$ and
$x \in \mathbb R - \{0\}$. The action of $\varphi$ on $g$ is the
restriction of the right-action of $\GL_4(\mathbb R)$, which is given by
$g \cdot \varphi = \varphi^T g \varphi$. In particular, if we take
$A = I_2$ and $x = 1$ in \eqref{eq:14}, there exists a unique
$B \in \mathbb R^{2 \times 2}$ such that $g \cdot \varphi$ is block
diagonal. More precisely, $B = -R^{-1}Q^T$.

We can now make an element of $G$ with $B=0$ act on a block diagonal
representative of $g$ and, since $\GL_2(\mathbb R)$ is transitive on
$\Sym_2^+$, we conclude that every orbit of $G$ meets an element of
the form
\begin{equation*}
  \begin{pmatrix}
    I_2 & 0 \\
    0 & R
  \end{pmatrix}
\end{equation*}
with $R \in \Sym_2^+$. Now, to fully determine the action of $G$ on
$\Sym_4^+$ we only need to look at the action of
$\OO(2) \times \GL_1(\mathbb R)$ on $\Sym_2^+$, since any element of
$G$ which leaves a representative of $g$ of the previous form must
have $A\in \OO(2)$. Note that this action is the one given as follows
\begin{equation*}
  R \cdot (A, x) = x^2A^TRA = x^2A^{-1}RA.
\end{equation*}
On the other hand, $R$ is conjugated by an orthogonal matrix to a
diagonal matrix. So one can choose $x$ in such a way that the first
diagonal element of the conjugated matrix is $1$. Hence we obtain the
following result.

\begin{lemma}
  \label{sec:left-invar-metr}
  Every orbit of $G$ on $\Sym_4^+$ intersects exactly once the
  subset
  \begin{equation*}
    \{\diag(1,1,1,r): 0 < r \le 1\}.
  \end{equation*}
\end{lemma}

\begin{proof}
  Let $r > 0$ and denote $g_r = \diag(1, 1,1, r)$.  We have seen that
  each orbit has an element of the form $g_r$. Suppose that there
  exists $\varphi \in G$ and $r' > 0$ such that
  $g_r \cdot \varphi = g_{r'}$. Assuming that $\varphi$ has the form
  \eqref{eq:14}, a simple calculation shows that
  \begin{equation*}
    g_r \cdot \varphi =
    \begin{pmatrix}
      A^TA + B^T \diag(1, r) B & x B^T \diag(1, r) \sigma(A) \\
      x \sigma(A)^T \diag(1, r) B & x^2 \sigma(A)^T \diag(1, r)
      \sigma(A)
    \end{pmatrix}
  \end{equation*}

  It follows that $B = 0$ and $A \in \OO(2)$. Without loss of
  generality, we can assume that $A \in \SO(2)$ rotates an angle
  $\theta$ around the origin. Now
  \begin{equation*}
    x^2 \sigma(A)^T
    \begin{pmatrix}
      1 & 0 \\
      0 & r
    \end{pmatrix}
    \sigma(A) = x^2
    \begin{pmatrix}
      \cos^2\theta + r \sin^2\theta & (1 - r) \sin \theta
      \cos \theta \\
      (1 - r) \sin \theta \cos\theta & \sin^2\theta + r
      \cos^2\theta
    \end{pmatrix}.
  \end{equation*}
  So, the only possibilities for $g_r \cdot \varphi = g_{r'}$ are
  $r = 1$, which implies $r' = 1$; $\sin \theta = 0$, which implies
  $r' = r$; and $\cos \theta = 0$ which implies $r' = 1/r$. From this
  the lemma follows.
\end{proof}

\begin{remark}
  Recall that in the symmetric space
  $\Sym^{+}_4= \GL_4(\mathbb R)/{\OO(4)}$ the symmetry, at an element
  $p \in \Sym^+_4$ is given by $s_p(q) = p q^{-1} p$. Set
  $S' = \{\diag(1, 1, 1, r): r \in \mathbb{R}^+\}$. Then it is
  straightforward to see that $s_p(S')=S'$ for each $p\in S'$ and
  hence, $S'$ is a totally geodesic submanifold.  In fact, if $\alpha$
  denotes the second fundamental form of $S'$, for each $p\in S'$ and
  $v,w\in S'$ we have that
  $-\alpha(v, w) = (ds_p)_p(\alpha(v,w)) = \alpha((ds_p)_p v,(ds_p)_p
  w) = \alpha(v, w)$. So $\alpha \equiv 0$.
\end{remark}

\begin{theorem}
  The moduli space $\mathcal M(H_4)/{\sim}$ of left-invariant metrics
  on $H_4$ up to isometric automorphism is homeomorphic to the space
  \begin{equation*}
    (0, 1] \times \Sym_2^+/\mathbb Z_2,
  \end{equation*}
  where $\mathbb Z_2$ is the subgroup of $\I(\Sym_2^+)$ generated by
  $\sigma
  \begin{pmatrix}
    a & b \\
    b & c
  \end{pmatrix}
  =
  \begin{pmatrix}
    a & -b \\
    -b & c
  \end{pmatrix}
  $. Moreover, every left-invariant metric is conjugated by an
  automorphism to a unique metric of the form
  \begin{equation}
    \label{eq:15}
    g = \sum_{i = 1}^3 e^i \otimes e^i + r e^4
    \otimes e^4 + a e^5 \otimes e^5 + 2b e^5 \otimes e^6 + c e^6 \otimes e^6
  \end{equation}
  where $0 < r \le 1$, $a, b, c \ge 0$ and $ac - b^2 > 0$.
\end{theorem}

\begin{proof}
  Let $g$ be a left-invariant metric on $H_4$. Identify $g$ with the
  inner product on $\mathfrak h_4$ which in the canonical basis is
  represented by the matrix $
  \begin{pmatrix}
    P & Q \\
    Q^T & R
  \end{pmatrix}
  $ where $P \in \Sym_4^+$, $R \in \Sym_2^+$ and
  $Q \in \mathbb R^{4 \times 2}$. With a similar argument as the one
  given for Lemma \ref{sec:left-invar-metr} one can assume that
  $Q = 0$ and $P = \diag(1, 1, 1, r)$ with $0 < r \le 1$. Denote
  $g = g_{r, R}$ to indicate that, up to automorphism, $g$ only
  depends on $0 < r \le 1$ and $R \in \Sym_2^+$. Let $\varphi$ be an
  automorphism of $\mathfrak h_4$ and let us write $\varphi$ in the
  canonical basis as
  \begin{equation}
    \label{eq:16}
    \begin{pmatrix}
      A & 0 & 0 \\
      B & x\sigma(A) & 0 \\
      M_1 & M_2 & \Delta(A, B, x)
    \end{pmatrix}
  \end{equation}
  (see Theorem \ref{sec:autom-group-mathfr-6}). As it follows from
  the proof of Lemma \ref{sec:left-invar-metr},
  $g_{r, R} \cdot \varphi= g_{r', R'}$ if and only if $A \in \OO(2)$,
  $x = \pm 1$ and $B = M_1 = M_2 = 0$. So,
  $\Delta(A, B, x) \in \{I_2, \diag(1, -1)\} \cup \{-I_2, \diag(-1,
  1)\}$. Since $-I_2$ acts trivially on $\Sym_2^+$, we can assume
  that $\Delta(A, B, x) \in \{I_2, \diag(1, -1)\} \simeq \mathbb Z_
  2$. Since conjugation by $\diag(1, -1)$ acts as the involution
  $\sigma$, we conclude that any left-invariant metric is equivalent
  to one of the form $g_{r, R}$, and such a metric is unique if we
  require $0 < r \le 1$ and that all the entries on $R$ are non
  negative.
\end{proof}

\begin{corollary}
  Let $g_{r, a, b, c}$ be the left-invariant metric on $H_4$ given in
  \eqref{eq:15}. Then the full isometry group of $g_{r, a, b, c}$ is
  given by
  \begin{equation*}
    \I(H_4, g_{r, a, b, c}) =
    \begin{cases}
      H_4 \rtimes (\OO(2) \rtimes \mathbb Z_2) & r = 1 \text{ and } b = 0 \\
      H_4 \rtimes \OO(2) & r = 1 \text{ and } b \neq 0 \\
      H_4 \rtimes (\mathbb Z_2 \times \mathbb Z_2) & r \neq 1 \text{
        and } b = 0 \\
      H_4 \rtimes \mathbb Z_2 & r \neq 1 \text{ and } b \neq 0
    \end{cases}
  \end{equation*}
\end{corollary}

\begin{proof}
  We use the same argument as in the proof of Corollary
  \ref{sec:left-invar-metr-1}. A generic automorphism $\varphi$ of
  $\Aut(\mathfrak h_4)$ can be written as \eqref{eq:16}. Recall that
  if $\varphi$ preserves the metric then $B$, $M_1$ and $M_2$ must
  vanish and $A \in \OO(2)$. This implies that
  \begin{equation*}
    \Delta(A, B, x) = \Delta(A, x) =
    \begin{pmatrix}
      \varepsilon & 0 \\
      0 & x \varepsilon
    \end{pmatrix}
  \end{equation*}
  with $\varepsilon \in \{\pm 1\}$. Hence $|x| = 1$ and $x = -1$ is
  only possible if $b = 0$. Since
  $\sigma: \GL_2(\mathbb R) \to \GL_2(\mathbb R)$ leaves $\OO(2)$
  invariant, if $r \neq 1$, then $A \in \{\pm I_2\}$. From the
  previous comments the corollary follows.
\end{proof}

\section{The case of $\mathfrak h_2 = (0, 0, 0, 0, 12, 34)$}

\subsection{Automorphism group}

Let $e_1, \ldots, e_6$ be the basis of $\mathfrak h_2$ such that the
only non trivial brackets are
\begin{align*}
  [e_1, e_2] = -e_5, && [e_3, e_4] = -e_6.
\end{align*}
Clearly, $\mathfrak h_2$ is isomorphic to the direct product of two
copies of the $3$-dimensional Heisenberg Lie algebra $\heis_1$. Recall
that the only ideals of $\mathfrak h_2$ isomorphic to $\heis_1$ are
the ones corresponding to factors in the decomposition
\begin{equation}
  \label{eq:18}
  \mathfrak h_2 \simeq \heis_1 \oplus \heis_1
\end{equation}
modulo $[\mathfrak h_2, \mathfrak h_2]$. More precisely, $\mathfrak k$
is such an ideal if and only if
\begin{equation*}
  \mathfrak k = \spann_{\mathbb R}\{e_1 + Z_1, e_2 + Z_2, e_5\} \qquad
  \text{or} \qquad   \mathfrak k = \spann_{\mathbb R}\{e_3 + Z_1, e_4 + Z_2, e_6\}
\end{equation*}
where $Z_1, Z_2$ are two fixed elements in
$[\mathfrak h_2, \mathfrak h_2]$. In fact, let $\mathfrak k$ be an
ideal of $\mathfrak h_2$ isomorphic to $\heis_1$. There must exist at
least one element $x\in \mathfrak k$ such that $e^i(x)\neq 0$ for some
$i=1,\ldots,4$. This implies that either $e_5$ or $e_6$ belongs to
$\mathfrak{k}$. Assume first that $e_5\in \mathfrak{k}$. Since the
center of $\mathfrak k$ is one dimensional, $e_6\notin \mathfrak{k}$
and so $e^3(\mathfrak k) = e^4(\mathfrak k) = 0$. The other case is
analogous. Now if $\varphi: \mathfrak h_2 \to \mathfrak h_2$ is an
automorphism, then the induced linear map
$\tilde \varphi: \mathfrak h_2/[\mathfrak h_2, \mathfrak h_2] \to
\mathfrak h_2/[\mathfrak h_2, \mathfrak h_2]$ either preserves or
swaps the factors of the decomposition
$\mathfrak h_2/[\mathfrak h_2, \mathfrak h_2] \simeq \spann_{\mathbb
  R}\{e_1, e_2\} \oplus \spann_{\mathbb R}\{e_3, e_4\}$. Notice that
the involution $\varphi_0: \mathfrak h_2 \to \mathfrak h_2$, which is
given by
\begin{equation}
  \label{eq:20}
  \begin{pmatrix}
    0 & 0 & 1 & 0 & 0 & 0 \\
    0 & 0 & 0 & 1 & 0 & 0 \\
    1 & 0 & 0 & 0 & 0 & 0 \\
    0 & 1 & 0 & 0 & 0 & 0 \\
    0 & 0 & 0 & 0 & 0 & 1 \\
    0 & 0 & 0 & 0 & 1 & 0
  \end{pmatrix},
\end{equation}
in the canonical basis, is an automorphism of $\mathfrak h_2$ that
reverses the decomposition \eqref{eq:18}.

\begin{theorem}
  \label{sec:automorphism-group-1}
  There exists an isomorphism of Lie groups
  \begin{equation}
    \label{eq:19}
    \Aut(\mathfrak h_2) \simeq \mathbb R^8 \rtimes ((\GL_2(\mathbb R)
    \times \GL_2(\mathbb R)) \rtimes \mathbb Z_2).
  \end{equation}
  More precisely, every automorphism of $\mathfrak h_2$ can be
  represented in the canonical basis by a matrix of the form
  \begin{equation}
    \label{eq:21}
    \begin{pmatrix}
      A & 0 & 0 \\
      0 & B & 0 \\
      M_1 & M_2 & \Delta(A, B)
    \end{pmatrix}
    \qquad \text{or} \qquad
    \begin{pmatrix}
      0 & A & 0 \\
      B & 0 & 0 \\
      M_1 & M_2 & \Delta'(A, B)
    \end{pmatrix}
  \end{equation}
  where $A, B \in \GL_2(\mathbb R)$, $M_1, M_2 \in \mathbb R^{2 \times
    2}$, and
  \begin{equation*}
    \Delta(A, B) =
    \begin{pmatrix}
      \det A & 0 \\
      0 & \det B
    \end{pmatrix},
    \qquad \qquad
    \Delta'(A, B) =
    \begin{pmatrix}
      0 & \det A \\
      \det B & 0
    \end{pmatrix}.
  \end{equation*}

  In particular, $\Aut(\mathfrak h_2)$ has $8$ connected components
  and
  \begin{equation*}
    \Aut(\mathfrak h_2) / {\Aut_0(\mathfrak h_2)} \simeq (\mathbb Z_2
    \times \mathbb Z_2) \rtimes \mathbb Z_2 \simeq \mathbb Z_4 \rtimes
    \mathbb Z_2
  \end{equation*}
  is isomorphic to the dihedral group $D_4$.
\end{theorem}

Recall the following identifications. The subgroup $\mathbb Z_2$ is
identified with the subgroup of $\Aut(\mathfrak h_2)$ generated by
$\varphi_0$ in \eqref{eq:20}. The subgroups isomorphic to
$\GL_2(\mathbb R)$ correspond to the automorphisms with
$M_1 = M_2 = 0$ and $B = I_2$ or $A = I_2$. Finally the normal
subgroup isomorphic to $\mathbb R^8$ is obtained in the connected
component of the identity with $A = B = I_2$.

\begin{proof}
  [{\proofname} of Theorem \Ref{sec:automorphism-group-1}] It follows
  from the discussion at the beginning of this subsection.  In fact,
  any Lie algebra automorphism $\mathfrak h_2 \to \mathfrak h_2$,
  which preserves or swaps the factors of the decomposition
  \eqref{eq:18} modulo $[\mathfrak h_2, \mathfrak h_2]$, has one of
  the forms described in \eqref{eq:21} and it is easy to verify that
  all of these maps are automorphisms.
\end{proof}

\subsection{Left-invariant metrics}

We follow the same approach as in the previous cases, so let us first
study the action of
$(\GL_2(\mathbb R) \times \GL_2(\mathbb R)) \rtimes \mathbb Z_2$ on
$\Sym_4^+ = \GL_4(\mathbb R)/{\OO(4)}$. We do not lose generality by
considering the action of the diagonal subgroup
$\GL_2(\mathbb R) \times \GL_2(\mathbb R) \subset \GL_4(\mathbb
R)$. Recall that if we write a generic element $g \in \Sym_4^+$ as
\begin{equation*}
  \begin{pmatrix}
    P & Q \\
    Q^T & R
  \end{pmatrix},
\end{equation*}
with $P, R \in \Sym_2^+$ and $Q \in \mathbb R^{2 \times 2}$, then the
above action is given by
\begin{equation*}
  g \cdot (A, B) =
  \begin{pmatrix}
    A^T P A & A^T Q B \\
    B^T Q^T A & B^T R B
  \end{pmatrix}
\end{equation*}
and so the orbit of every $g$ meets an element of the form $g =
\begin{pmatrix}
  I_2 & Q \\
  Q^T & I_2
\end{pmatrix}
$. Therefore we can restrict our attention to the action of $\OO(2)
\times \OO(2)$ on $\mathbb R^{2 \times 2}$ given by
\begin{equation}
  \label{eq:23}
  Q \cdot (A, B) = A^T Q B.
\end{equation}

Recall that the positive definite inner product
\begin{equation}
  \label{eq:22}
  \langle Q_1, Q_2\rangle = \frac12\tr(Q_1Q_2^T)
\end{equation}
makes $\mathbb R^{2 \times 2}$ an Euclidean space and moreover, the
action \eqref{eq:23} is isometric.

\begin{lemma}
  \label{sec:left-invar-metr-3}
  Let $Q \in \mathbb R^{2 \times 2}$ and let $\mathcal O_Q$ be the
  orbit of $Q$ under the action of $\OO(2) \times \OO(2)$ given by
  $Q \cdot (A, B) = A^T Q B$. Then:
  \begin{enumerate}
  \item \label{item:10} $\mathcal O_Q$ intersect the subspace of
    diagonal matrices.
  \item \label{item:11}  Moreover, $\mathcal O_Q$ contains exactly one
    element of the form $\diag(a, b)$ with $0 \le a \le b$.
  \end{enumerate}
\end{lemma}

\begin{proof}
  The decomposition $\mathbb R^{2 \times 2} = \so(2) \oplus \Sym_2$ is
  orthogonal with respect to the metric given in \eqref{eq:22}. Let us
  consider first the isometric action of $\SO(2)$ on
  $\mathbb R^{2 \times 2}$ given by the restriction to the connected
  component of the first factor: $Q \cdot A = A^T Q$ and let
  $\mathcal O'_Q$ be the orbit of $Q$ under this action. It follows
  that $\mathcal O'_Q \cap \Sym_2 \neq \varnothing$. In fact, we can
  assume, by multiplying by a multiple of $I_2$ that $\|Q\| =
  1$. Hence $\mathcal O'_Q$ is a great circle in the unit sphere
  $S^3 \subset \mathbb R^{2 \times 2}$ and it must intersect the
  $3$-dimensional subspace $\Sym_2$. Item \eqref{item:10} follows
  by noticing that if $Q \in \Sym_2$, then there is $A \in \OO(2)$
  such that $A^T Q A$ is diagonal.

  Now suppose that $\diag(a, b) \in \mathcal O_Q$. This implies that
  $\diag(\varepsilon_1 a, \varepsilon_2 b)$ and
  $\diag(\varepsilon_1 b, \varepsilon_2 a)$ also belong to
  $\mathcal O_Q$, for any
  $\varepsilon_1, \varepsilon_2 \in \{\pm 1\}$. So, $\mathcal O_Q$ has
  an element of the form $\diag(a, b)$ with $0 \le a \le b$. Suppose
  that there exists $A, B \in \OO(2)$ such that
  $A^T \diag(a, b) B = \diag(a', b')$, for some $0 \le a' \le
  b'$. Since the action is isometric we can assume that
  $a^2 + b^2 = (a')^2 + (b')^2 = 1$. Moreover, we do not lose
  generality by assuming that $A, B \in \SO(2)$, say
  \begin{align}
    \label{eq:25}
    A =
    \begin{pmatrix}
      \cos t & -\sin t \\
      \sin t & \cos t
    \end{pmatrix}, && B =
                      \begin{pmatrix}
                        \cos s & -\sin s \\
                        \sin s & \cos s
                \end{pmatrix},
  \end{align}
  which yields to the equations
  \begin{align*}
    a \cos t & = a' \cos s, &
    b \sin t & = a' \sin s, \\
    a \sin t & = b' \sin s, &
    b \cos t & = b' \cos s.
  \end{align*}
  We can further assume that
  $a, b, a', b', \cos t, \sin t, \cos s, \sin s$ are all non zero,
  otherwise the result holds trivially. This implies $ab' = ba'$ and
  $aa' = bb'$. From this it is easy to see that $a = b$, and hence
  $a' = b'$, which proves \eqref{item:11}.
\end{proof}

\begin{remark}
  Observe that the right $(\OO(2) \times \OO(2))$-action on
  $\mathfrak{gl}_2(\mathbb R)$ of Lemma~\ref{sec:left-invar-metr-3}
  coincides with the isotropy representation of the symmetric space
  $\OO(2,2) / (\OO(2) \times \OO(2))$, i.e., the Grassmannian of
  positive definite $2$-planes in $\mathbb{R}^4$ with the metric of
  signature $2$. This readily implies that there must always exist a
  section given by the diagonal matrices.
\end{remark}
\begin{remark}
  We can also give a geometric argument for the proof of part
  \eqref{item:11} of Lemma~\ref{sec:left-invar-metr-3}. Consider the
  geodesics $\gamma(t) = A^T \diag(a, b)$ and
  $\beta(s) = \diag(a', b') B^T$ of
  $S^3 \subset \mathbb R^{2 \times 2}$, where $A$ and $B$ are as in
  (\ref{eq:25}). The image of $\gamma(t)$ is the intersection of $S^3$
  with the plane $\pi_1 \subset \mathbb R^{2 \times 2}$ generated by
  $\diag(a, b)$ and $
  \begin{pmatrix}
    0 & b \\
    -a & 0
  \end{pmatrix}
  $. Similarly, the image of $\beta(s)$ is the intersection of $S^3$
  with the plane $\pi_2$ generated by $ \diag(a', b')$ and $
  \begin{pmatrix}
    0 & a' \\
    -b' & 0
  \end{pmatrix}
  $. So, assuming $0 \le a \le b$, $0 \le a' \le b'$ and
  $a^2 + b^2 = (a')^2 + (b')^2 = 1$, the condition
  $(a, b) \neq (a', b')$ implies $\pi_1 \cap \pi_2 = \{0\}$ and
  therefore $A^T\diag(a, b) \neq \diag(a', b') B^T$ for all
  $A, B \in \OO(2)$.
\end{remark}

\begin{theorem}
  \label{sec:left-invar-metr-5}
  Let $H_2$ be the simply connected Lie group with Lie algebra
  $\mathfrak h_2$.  The moduli space $\mathcal M(H_2)/{\sim}$ of
  left-invariant metrics on $H_2$ up to isometric automorphism is
  homeomorphic to the space
  \begin{equation*}
    \{(a, b) \in \mathbb R^2: 0 \le a \le b < 1\} \times \Sym_2^+/\mathbb Z_2,
  \end{equation*}
  where $\mathbb Z_2$ is the subgroup of $\I(\Sym_2^+)$ generated by
  the involution
  \begin{equation*}
    \sigma
    \begin{pmatrix}
      E & F \\
      F & G
    \end{pmatrix}
    =
    \begin{pmatrix}
      E & -F \\
      -F & G
    \end{pmatrix}.
  \end{equation*}
  Moreover, every left-invariant metric on $H_2$ is conjugated by an
  automorphism to a unique metric of the form
  \begin{equation}
    \label{eq:26}
    g = \sum_{i = 1}^4 e^i \otimes e^i + 2a e^1 \otimes e^3 + 2b
    e^2 \otimes e^4 + E e^5 \otimes e^5 + 2F e^5 \otimes e^6 + G e^6
    \otimes e^6,
  \end{equation}
  where $0 \le a \le b$, $E, F, G \ge 0$ and $EG - F^2 > 0$.
\end{theorem}

\begin{proof}
  From Theorem \ref{sec:automorphism-group-1},
  $\Aut_0(\mathfrak h_2) \simeq \mathbb R^8 \rtimes (\GL_2(\mathbb R)
  \times \GL_2(\mathbb R))_0$, and hence we can use a similar argument
  as in Lemma \ref{sec:left-invar-metr-6} to conclude that any
  left-invariant metric on $H_2$ is equivalent to a block diagonal
  metric $g = \diag(g_1, g_2)$ with $g_1 \in \Sym_4^+$ and
  $g_2 \in \Sym_2^+$. From the discussion at the beginning of this
  subsection and Lemma \ref{sec:left-invar-metr-3} we know that there
  exist unique $0 \le a \le b < 1$ such that $g_1$ is equivalent under
  the action of $\GL_2(\mathbb R) \times \GL_2(\mathbb R)$ to
  \begin{equation}
    \label{eq:27}
    \begin{pmatrix}
      1 & 0 & a & 0 \\
      0 & 1 & 0 & b \\
      a & 0 & 1 & 0 \\
      0 & b & 0 & 1
    \end{pmatrix}.
  \end{equation}
  Recall that we need to impose the additional condition $b < 1$ in
  order to get that the matrix \eqref{eq:27} is positive definite. To
  complete the proof of the theorem, one only needs to note that the
  isotropy of $(\GL_2(\mathbb R) \times \GL_2(\mathbb R))_0$ is
  isomorphic to $\mathbb Z_ 2 \times \mathbb Z_2$. This group induces,
  after making the action effective, the action of $\mathbb Z_2$ on
  $\Sym_2^+$ given by the involution $\sigma$.
\end{proof}

\begin{corollary}
  Let $g = g_{a, b, E, F, G}$ be the left-invariant metric on $H_2$
  defined in (\ref{eq:26}). Then the full isometry group of $g$ is
  given by
  \begin{equation*}
    \I(H_2,g) =
    \begin{cases}
      H_2 \rtimes ((\OO(2) \times \OO(2)) \rtimes \mathbb Z_2), & a =
      b = 0, \, F = 0, \, E = G \\
      H_2 \rtimes (\OO(2) \times \OO(2)), & a = b = 0, \, F = 0, \, E
      \neq G \\
      H_2 \rtimes (\mathrm{S}(\OO(2) \times \OO(2)) \rtimes \mathbb
      Z_2), & a =
      b = 0, \, F \neq 0, \, E = G \\
      H_2 \rtimes \mathrm{S}(\OO(2) \times \OO(2)), & a = b = 0, \, F
      \neq 0, \,
      E \neq G \\
      H_2 \rtimes (\diag(\OO(2) \times \OO(2)) \rtimes \mathbb Z_2), &
      a = b \neq 0, \, E = G \\
      H_2 \rtimes \diag(\OO(2) \times \OO(2)), & a = b \neq 0, \, E
      \neq G \\
      H_2 \rtimes D_4, & 0 \le a < b, \, E = G \\
      H_2 \rtimes (\mathbb Z_2 \times \mathbb Z_2), & 0 \le a < b, E
      \neq G. \\
    \end{cases}
  \end{equation*}
\end{corollary}

\begin{proof}
  From Theorem \ref{sec:left-invar-metr-5}, we can identify the
  left-invariant metric $g = g_{a, b, E, F, G}$ with the symmetric
  positive definite matrix
  \begin{equation*}
    \label{eq:28}
    g =
    \begin{pmatrix}
      1 & 0 & a & 0 & 0 & 0 \\
      0 & 1 & 0 & b & 0 & 0 \\
      a & 0 & 1 & 0 & 0 & 0 \\
      0 & b & 0 & 1 & 0 & 0 \\
      0 & 0 & 0 & 0 & E & F \\
      0 & 0 & 0 & 0 & F & G
    \end{pmatrix}.
  \end{equation*}
  On the other hand, from Theorem \ref{sec:automorphism-group-1}, the
  discrete group $\Aut(\mathfrak h_2) / {\Aut_0(\mathfrak h_2)}$ is
  isomorphic to the dihedral group
  $D_4 \simeq (\mathbb Z_2 \times \mathbb Z_2) \rtimes \mathbb Z_2$,
  where each $\mathbb Z_2$ it is generated by the projection of the
  involutive automorphisms given by
  \begin{align*}
    \varphi_1 & = \diag(-1,1,1,1,-1,1) \\
    \varphi_2 & = \diag(1,1,-1,1,1,-1) \\
    \varphi_3 & =
        \begin{pmatrix}
      0 & 0 & 1 & 0 & 0 & 0 \\
      0 & 0 & 0 & 1 & 0 & 0 \\
      1 & 0 & 0 & 0 & 0 & 0 \\
      0 & 1 & 0 & 0 & 0 & 0 \\
      0 & 0 & 0 & 0 & 0 & 1 \\
      0 & 0 & 0 & 0 & 1 & 0
    \end{pmatrix}
  \end{align*}
  Moreover, $\varphi_1$ and $\varphi_2$ generate the two connected
  components on each factor of $\OO(2) \times \OO(2)$, after their
  natural inclusion into $\Aut(\mathfrak h_2)$, and $\varphi_3$ gives
  the bijection between the block diagonal and anti-diagonal
  automorphisms described in Theorem \ref{sec:automorphism-group-1}.

  The following facts are easy to verify:
  \begin{enumerate}
  \item $g \cdot (\varphi_1 \varphi_2) = g \cdot (\varphi_2 \varphi_1)
    = g$;
  \item $g \cdot \varphi_1 = g$ if and only if $a = 0$ and $F = 0$;
  \item $g \cdot \varphi_2 = g$ if and only if $a = 0$ and $F = 0$;
  \item $g \cdot \varphi_3 = g$ if and only if $E = G$.
  \end{enumerate}

  So, in order to compute all the isometric automorphisms we can
  restrict our attention to the action of the connected component
  $\SO(2) \times \SO(2)$. This raises three possibilities. Firstly, if
  $a = b = 0$, then it is clear that $g \cdot \varphi = g$ for all
  $\varphi \in \SO(2) \times \SO(2)$. Secondly, if $a = b \neq 0$,
  and $\varphi = (A, B)$ is such that $g \cdot \varphi = A^{-1} g B$,
  then $A = B$. Finally, if $a < b$, the only $\varphi \in \SO(2)
  \times \SO(2)$ such that $g \cdot \varphi$ are $\varphi = \pm
  I_6$. This completes the proof of the corollary.
\end{proof}

\section{The case of $\mathfrak h_9 = (0, 0, 0, 0, 12, 14 + 25)$}

This is the most difficult case to describe, since as we will see,
$\Aut(\mathfrak h_9)$ does not admit a normal abelian subgroup such
that $\Aut(\mathfrak h_9)$ is the semi-direct product of this subgroup
and an algebraic subgroup which descends down to the quotient
$\mathfrak h_9/[\mathfrak h_9, \mathfrak h_9]$. According to our
notation $\mathfrak h_9$, has a basis $e_1, \ldots, e_6$ such that
$de^1 = \cdots = de^4 = 0$, $de^5 = e^{12}$ and
$de^6 = e^{14} + e^{25}$, or equivalently, the non trivial brackets
are
\begin{align*}
  [e_1, e_2] = -e_5 && [e_1, e_4] = [e_2, e_5] = -e_6.
\end{align*}
In particular, $\mathfrak h_9$ is $3$-step nilpotent. We find
it convenient change to the basis $\hat e_1 = e_2$, $\hat e_2 = e_1$,
$\hat e_3 = e_4$, $\hat e_4 = e_3$, $\hat e_5 = e_5$ and
$\hat e_6 = e_6$ where the non trivial brackets are
\begin{align}
  \label{eq:29}
  [\hat e_1, \hat e_2] = \hat e_5 && [\hat e_1, \hat e_5] = [\hat e_2,
                                     \hat e_3] = -\hat e_6.
\end{align}

Notice that with respect to this basis we have
$\mathfrak z(\mathfrak h_9) = \spann_{\mathbb R}\{\hat e_4, \hat
e_6\}$ and
$[\mathfrak h_9, \mathfrak h_9] = \spann_{\mathbb R}\{\hat e_5, \hat
e_6\}$.

\subsection{Automorphism group}

\begin{theorem}
  \label{sec:automorphism-group-2}
  With respect to the basis $\hat e_1, \ldots, \hat e_6$ every
  automorphism $\varphi \in \Aut(\mathfrak h_9)$ has the form
  \begin{equation}
    \label{eq:30}
    \begin{pmatrix}
      a_{11} & 0 & 0 & 0 & 0 & 0 \\
      a_{21} & a_{22} & 0 & 0 & 0 & 0 \\
      a_{31} & a_{32} & a_{11}^{2} & 0 & 0 & 0 \\
      a_{41} & a_{42} & a_{43} & a_{44} & 0 & 0 \\
      a_{51} & a_{52} & -a_{11} a_{21} & 0 & a_{11} a_{22} & 0 \\
      a_{61} & a_{62} & a_{63} & a_{64} & a_{22} a_{31} - a_{21}
      a_{32} - a_{11} a_{52} & a_{11}^{2} a_{22}
    \end{pmatrix}.
\end{equation}

In particular, $\Aut(\mathfrak h_9)$ is a $15$-dimensional solvable
Lie group, which has $8$ connected components and
$\Aut(\mathfrak h_9) / {\Aut_0(\mathfrak h_9)} \simeq \mathbb Z_2
\times \mathbb Z_2 \times \mathbb Z_2$.
\end{theorem}

\begin{proof}
  We show first that the matrix $(a_{ij})$ of $\varphi$ in the given
  basis is lower triangular.  It follows from \eqref{eq:29} that
  $\hat e^1(\varphi(\hat e_2)) = 0$. In fact,
  $(\ad_{\varphi(\hat e_2)})^2 = (\ad_{\hat e_2})^2 = 0$ and so the
  $\hat e_1$-component of $\varphi(\hat e_2)$ must vanish. Similarly,
  since
  $\dim(\ker \ad_{\varphi(\hat e_3)}) = \dim(\ker \ad_{\hat e_i}) =
  5$, then
  $\hat e^1(\varphi(\hat e_3)) = \hat e^2(\varphi(\hat e_3)) =
  0$. Recall that $\varphi$ preserves the subalgebras
  $\mathfrak z(\mathfrak h_9)$, $[\mathfrak h_9, \mathfrak h_9]$,
  $\mathfrak z(\mathfrak h_9) + [\mathfrak h_9, \mathfrak h_9]$ and
  $\mathfrak z(\mathfrak h_9) \cap [\mathfrak h_9, \mathfrak
  h_9]$. This implies that $\hat e^i(\varphi(\hat e_j)) = 0$ for all
  $i = 1, 2, 3$, $j = 4, 5, 6$; $\hat e^4(\varphi(\hat e_j)) = 0$ for
  $j = 5, 6$, and $\hat e^5(\varphi(\hat e_6)) = 0$. Moreover, since
  $\hat e^4 \in \mathfrak z(\mathfrak h_9)$, one has
  $\hat e^5(\varphi(\hat e_4)) = 0$.

  On the other side,
  \begin{align*}
    \hat e^5(\varphi(\hat e_5)) & = \hat e^5([\varphi(\hat e_1),
                                  \varphi(\hat e_2)]) = a_{11} a_{22} \\
    \hat e^6(\varphi(\hat e_5)) & = \hat e^6([\varphi(\hat e_1),
                                  \varphi(\hat e_2)]) = a_{22} a_{31}
                                  - a_{21} a_{32} - a_{11} a_{52},
  \end{align*}
  and since
  $a_{11}^2a_{22} = [\varphi(\hat e_1), \varphi(\hat e_5)] =
  [\varphi(\hat e_2), \varphi(\hat e_3)] = -\varphi(\hat e_6)$ we
  conclude that
  \begin{align*}
    \hat e^3(\varphi(\hat e_3)) & = a_{11}^2, &
    \hat e^5(\varphi(\hat e_3)) & = -a_{11} a_{21}, &
    \hat e^6(\varphi(\hat e_6)) & = a_{11}^2 a_{22}.
  \end{align*}

  This proves that $\varphi$ has the form \eqref{eq:30}. It is easy to
  see that any linear map of this form is an automorphism of
  $\mathfrak h_9$. (For instance, one can compute the dimension of
  $\Der(\mathfrak h_9)$ and check that it equals $15$. The
  automorphisms with $a_{11} = a_{22} = a_{44} = 1$ are the ones in the
  exponential of the nilradical of $\Der(\mathfrak h_9)$.)
\end{proof}

\begin{remark}
  Observe that from the previous theorem, $\Aut(\mathfrak h_9)$ is not
  the semi-direct product of an abelian normal subgroup.
\end{remark}

\subsection{Left-invariant metrics}

Let $T_6$ be the subgroup of $\GL_6(\mathbb R)$ of lower triangular
matrices, and denote by $T_6^+$ the normal subgroup of $T_6$ of
matrices whose diagonal entries are all positive. It is known that
$T_6^+$ acts simply transitively on $\Sym_6^+$ with the restriction of
the action of $\GL_6(\mathbb R)$. Moreover, this action is proper
since it is equivalent to the action by right multiplication of
$T_6^+$ on itself.

\begin{theorem}
  Let $H_9$ be the simply connected Lie group with Lie algebra
  $\mathfrak h_9$. The moduli space $\mathcal M(H_9)/{\sim}$ is a
  $6$-dimensional smooth manifold. Moreover, $\mathcal M(H_9)/{\sim}$
  is diffeomorphic to the homogeneous manifold
  $T_6^+/\Aut_0(\mathfrak h_9)$ and every left-invariant metric on
  $H_9$ is equivalent to a unique metric of the form
  \begin{align}\label{eq:32}
    g & = \hat e^1 \otimes \hat e^1 + \hat e^2 \otimes \hat e^2 + (A^2 + D^2) \hat e^3 \otimes
        \hat e^3 + DE \hat e^3 \otimes \hat e^4  \notag \\
      & \qquad + BD \hat e^3 \otimes \hat e^5 + (E^2 + 1)\hat e^4 \otimes\hat e^4 + BE
        \hat e^4 \otimes \hat e^5 + (B^2 + F^2)\hat e^5 \otimes\hat e^5  \\
      & \qquad\qquad + CF\hat e^5 \otimes\hat e^6 + C^2\hat e^6 \otimes\hat e^6, \notag
  \end{align}
  where $A, B, C > 0$, $D, E, F \in \mathbb R$, and
  $\hat e^1,\ldots,\hat e^6$ is the dual basis of
  $\hat e_1,\ldots,\hat e_6$.
\end{theorem}

\begin{proof}
  Since $\mathcal M(H_9) \simeq \Sym_6^+$ is connected, it is enough
  to consider the orbits of $\Aut_0(\mathfrak h_9)$ in $\Sym_6^+$. Let
  $\Phi: T_6^+ \to \Sym_6^+$ given by $\Phi(X) = X^tX$. As we mention
  above, the action of $\Aut_0(\mathfrak h_9)$ on $\Sym_6^+$ is
  equivalent via $\Phi$ to the action of $\Aut_0(\mathfrak h_9)$ by
  right multiplication. Note that the submanifold $\Sigma$ of
  $T_6^+$ given by
  \begin{equation}
    \label{eq:31}
    \Sigma = \left\{
      \begin{pmatrix}
        1 & 0 & 0 & 0 & 0 & 0 \\
        0 & 1 & 0 & 0 & 0 & 0 \\
        0 & 0 & A & 0 & 0 & 0 \\
        0 & 0 & 0 & 1 & 0 & 0 \\
        0 & 0 & D & E & B & 0 \\
        0 & 0 & 0 & 0 & F & C
      \end{pmatrix}: A, B, C > 0, \, D, E, F \in \mathbb R
    \right\}
  \end{equation}
  is a slice for the action of $\Aut_0(\mathfrak h_6)$. Moreover, the
  map $\Sigma \times \Aut_0(\mathfrak h_9) \to T_6^+$ given by
  $(S, \varphi) \mapsto S \varphi$ is a diffeomorphism. In fact, lets
  see that for any $X \in T_6^+$, there exist unique $S \in \Sigma$
  and $\varphi \in \Aut_0(\mathfrak h_9)$ such that $S \varphi =
  X$. Denote $X = (x_{ij})$ and assume that $\varphi$ and $S$ are as
  in \eqref{eq:30} and \eqref{eq:31} respectively. It is clear that
  $A, B, C$ and the elements on the diagonal of $\varphi$ are uniquely
  determined by $X$ and so we can assume that $A, B, C$ and all the
  elements on the diagonals of $X$ and $\varphi$ are equal to
  $1$. Moreover, since the principal $4 \times 4$ block of $S \varphi$
  coincides with the one of $\varphi$, we can assume further that
  $a_{ij} = x_{ij} = 0$ for all $2 \le i \le 4$ and $1 \le j \le
  3$. The equation $S \varphi = X$ has now the form
  \begin{equation*}
    \begin{pmatrix}
      1 & 0 & 0 & 0 & 0 & 0 \\
      0 & 1 & 0 & 0 & 0 & 0 \\
      0 & 0 & 1 & 0 & 0 & 0 \\
      0 & 0 & 0 & 1 & 0 & 0 \\
      a_{51} & a_{52} & D & E & 1 & 0 \\
      F a_{51} + a_{61} & F a_{52} + a_{62} & a_{63} & a_{64} & F - a_{52} & 1
    \end{pmatrix} =
    \begin{pmatrix}
      1 & 0 & 0 & 0 & 0 & 0 \\
      0 & 1 & 0 & 0 & 0 & 0 \\
      0 & 0 & 1 & 0 & 0 & 0 \\
      0 & 0 & 0 & 1 & 0 & 0 \\
      x_{51} & x_{52} & x_{53} & x_{54} & 1 & 0 \\
      x_{61} & x_{62} & x_{63} & x_{64} & x_{65} & 1
    \end{pmatrix}
  \end{equation*}
  which clearly has a unique solution in $S, \varphi$.

  So, $\Phi(\Sigma)$ is a full slice in $\mathcal M(H_9)$ for the
  action of the action of the automorphism group of $\mathfrak h_9$
  and every left-invariant metric on $H_9$ is equivalent to a unique
  metric which in the basis given by \eqref{eq:29} is represented by
  the matrix
  \begin{equation}
    \label{eq:70}
    \begin{pmatrix}
      1 & 0 & 0 & 0 & 0 & 0 \\
      0 & 1 & 0 & 0 & 0 & 0 \\
      0 & 0 & A^{2} + D^{2} & D E & B D & 0 \\
      0 & 0 & D E & E^{2} + 1 & B E & 0 \\
      0 & 0 & B D & B E & B^{2} + F^{2} & C F \\
      0 & 0 & 0 & 0 & C F & C^{2}
    \end{pmatrix},
  \end{equation}
  with $A, B, C > 0$, as we wanted to show.
\end{proof}

\begin{corollary}
  Let $g$ be the left-invariant metric on $H_9$ given by
  \eqref{eq:32}. Then the full isometry group of $g$ is given by
  \begin{equation*}
    \I(H_9, g) \simeq H_9\rtimes \mathbb{Z}_2^k,
  \end{equation*}
  where $k$ is the number of null parameters among $D, E$ and $F$.
\end{corollary}

\begin{proof}
  Let $\varphi$ be an automorphism of $\mathfrak{h}_9$ and let
  $\varphi = N + S$ be Jordan-Chevalley decomposition of $\varphi$,
  that is, $N$ is nilpotent and $S$ is semisimple such that $NS =
  SN$. Since $N$ and $S$ can be obtained as polynomials on $\varphi$,
  it follows from Theorem \ref{sec:automorphism-group-2} that, in the
  basis $\hat e_1, \ldots, \hat e_6$, $N$ is a strictly lower
  triangular matrix and $S$ is a lower triangular matrix, such that
  its diagonal elements coincide with the diagonal elements of
  $\varphi$. Since the isometric automorphisms of $H_9$ (which are
  induced by isometric automorphisms of $\mathfrak{h}_9$) constitute
  a compact subgroup of the isometry group of $H_9$, the matrix $S$
  must have all its diagonal entrances equal to $\pm 1$.

  Assume first that $S = I_6$. Since the coefficients
  $(N^k)_{i + 1, i} = 0$ for all $k \geq 2$, $i = 1, \ldots, 5$ and
  \begin{equation*}
    \varphi^n = \sum_{k = 0}^n \binom{n}{k} N^k,
  \end{equation*}
  we get that $(\varphi^n)_{i+1, i} = n N_{i+1, i}$. Since $\varphi^n$
  is an isometric automorphism (which must be bounded), we conclude
  that $N_{i+1,i}=0$ for all $i = 1, \ldots, 5$. Applying the same
  argument one can prove that $N=0$.

  If $S\neq I_6$, we consider the isometric automorphism
  $\varphi^2=S^2+2SN+N^2$. Since all the diagonal entrances of
  $\varphi^2$ are equal to $1$, its Jordan-Chevalley decomposition is
  given by $\varphi^2=N'+I_6$. We can apply the previous reasoning to
  $\varphi^2$ and conclude that $N'=0$. On the other hand, $S^2$ is
  semisimple and $2SN+N^2$ is nilpotent.  Hence $S^2=I_6$ and
  \begin{equation}
    \label{eqaux1}
    2SN+N^2=N'=0.
  \end{equation}
  But $SN$ and $N^2$ have minimal polynomials of different degrees,
  unless $N=0$. So, equation \eqref{eqaux1} holds only if $N=0$.

  We conclude that an isometric automorphism of $\mathfrak{h_9}$ must
  be a subgroup of
  \begin{equation*}
    \mathbb{Z}_2^3 = \{\diag(\varepsilon_1, \varepsilon_2,1,
    \varepsilon_3, \varepsilon_1 \varepsilon_2, \varepsilon_2):
    \varepsilon_1,\varepsilon_2,\varepsilon_3=\pm1\}.
  \end{equation*}

  Now a straightforward computation concludes the proof.
\end{proof}

\section{Hermitian metrics}
\label{sec:hermitian-metrics}

\subsection{Hermitian structures on $\mathfrak h_5$}
\label{sec:hermitian-iwasawa}

Recall that, from \cite{Di_Scala_2012}, every left-invariant metric on
$H_5$ is equivalent by an automorphism to one and only one metric
which is represented in the standard basis $e_1, \ldots, e_6$ by the
symmetric positive definite matrix
\begin{equation}
  \label{eq:33}
  g =
  \begin{pmatrix}
    1 & 0 & 0 & 0 & 0 & 0 \\
    0 & r & 0 & 0 & 0 & 0 \\
    0 & 0 & 1 & 0 & 0 & 0 \\
    0 & 0 & 0 & s & 0 & 0 \\
    0 & 0 & 0 & 0 & E & F \\
    0 & 0 & 0 & 0 & F & G
  \end{pmatrix},
\end{equation}
where $0 < s \le r \le 1$, $0 < E$, $0 \le F$, $0 < G$ and
$0 < E G - F^2$. It is not difficult to see that the left-invariant
almost Hermitian structures, with respect to $g$, which preserve the
orientation induced by the standard basis, are parameterized by two
$2$-spheres. More precisely, any orientation preserving,
left-invariant almost Hermitian structure (in the standard basis) has
either the form
\begin{equation}
  \label{eq:34}
  J_1 =
  \begin{pmatrix}
    0 & -a \sqrt{r} & -b & -c \sqrt{s} & 0 & 0 \\[.4pc]
    \frac{a}{\sqrt{r}} & 0 & -\frac{c}{\sqrt{r}} & \frac{b
      \sqrt{s}}{\sqrt{r}} & 0 & 0 \\[.4pc]
    b & c \sqrt{r} & 0 & -a \sqrt{s} & 0 & 0 \\[.4pc]
    \frac{c}{\sqrt{s}} & -\frac{b \sqrt{r}}{\sqrt{s}} &
    \frac{a}{\sqrt{s}} & 0 & 0 & 0 \\[.4pc]
    0 & 0 & 0 & 0 & -\frac{F}{\sqrt{-F^{2} + E G}} &
    -\frac{G}{\sqrt{-F^{2} + E G}} \\[.4pc]
    0 & 0 & 0 & 0 & \frac{E}{\sqrt{-F^{2} + E G}} &
    \frac{F}{\sqrt{-F^{2} + E G}}
  \end{pmatrix}
\end{equation}
or the form
\begin{equation}
  \label{eq:35}
  J_2 =
  \begin{pmatrix}
    0 & -a \sqrt{r} & -b & -c \sqrt{s} & 0 & 0 \\[.4pc]
    \frac{a}{\sqrt{r}} & 0 & \frac{c}{\sqrt{r}} & -\frac{b
      \sqrt{s}}{\sqrt{r}} & 0 & 0 \\[.2pc]
    b & -c \sqrt{r} & 0 & a \sqrt{s} & 0 & 0 \\[.4pc]
    \frac{c}{\sqrt{s}} & \frac{b \sqrt{r}}{\sqrt{s}} &
    -\frac{a}{\sqrt{s}} & 0 & 0 & 0 \\[.4pc]
    0 & 0 & 0 & 0 & \frac{F}{\sqrt{-F^{2} + E G}} &
    \frac{G}{\sqrt{-F^{2} + E G}} \\[.4pc]
    0 & 0 & 0 & 0 & -\frac{E}{\sqrt{-F^{2} + E G}} &
    -\frac{F}{\sqrt{-F^{2} + E G}}
  \end{pmatrix},
\end{equation}
where $(a, b, c) \in S^2$. So, we distinguish two cases in order to
determine when $J_1$ and $J_2$ are integrable.

\subsubsection{The case of $J_1$}

In order to simplify some calculations we denote
\begin{align*}
  \Delta = E G - F^2, && \alpha = \frac{\sqrt r + \sqrt s}{1 + \sqrt{r
                         s}}.
\end{align*}
With much patience and after long computations one can see that the
non-trivial equations on the integrability of $J_1$ are given by
\begin{align}
  a & = \frac{(1 - b^2) \sqrt \Delta}{G \alpha} \label{eq:37}, \\
  a & = \frac{(1 - c^2) \alpha \sqrt \Delta}{E} \label{eq:38}, \\
  Fa & = b c \sqrt \Delta \label{eq:39}, \\
  Fb & = \frac{Ec}{\alpha} - a c \sqrt \Delta \label{eq:40}, \\
  Fc & = G \alpha b - a b \sqrt \Delta \label{eq:41}.
\end{align}

It follows from (\ref{eq:37}) and (\ref{eq:38}) that $a > 0$ and, in
the generic case $F \neq 0$, (\ref{eq:39}) says that $b$ and $c$ must
have the same sign. Moreover, it is not hard to see that (\ref{eq:39})
follows from (\ref{eq:37}) and (\ref{eq:38}) by multiplying these two
equations. With a similar argument, one can see that (\ref{eq:40})
and (\ref{eq:41}) also follow from (\ref{eq:37}) and
(\ref{eq:38}). Now we can combine (\ref{eq:37}), (\ref{eq:38}) and the
condition $a^2 + b^2 + c^2 = 1$ to obtain the following quadratic
equation in $a$:
\begin{equation}
  \label{eq:42}
  a^2 - a \frac1{\sqrt \Delta} \left(\frac{E}{\alpha} + G
    \alpha\right) + 1 = 0.
\end{equation}
With a little algebra we see that the condition on (\ref{eq:42}) for
having real roots is equivalent to the tautology
$\left(\frac E \alpha - G \alpha\right)^2 + 4 F^2 \geq 0$. So the
solution between $0$ and $1$ is
\begin{equation}
  \label{eq:44}
  a = \frac12 \left( \frac1{\sqrt \Delta} \left(\frac{E}{\alpha} + G
      \alpha \right) - \sqrt{\frac 1 \Delta \left(\frac E \alpha  + G
        \alpha\right)^2 - 4}\right),
\end{equation}
which gives the following two values for $b$ and $c$:
\begin{align}
  b = \pm \sqrt{1 - \frac{G a \alpha}{\sqrt \Delta}},
  && c = \pm \sqrt{1 - \frac{E a}{\alpha \sqrt \Delta}}.
\end{align}

Finally, notice that when $F = 0$ equations (\ref{eq:39}),
(\ref{eq:40}) and (\ref{eq:41}) reduces to
\begin{equation}
  \label{eq:45}
  b c = \left(\frac{\sqrt E}{\sqrt G \alpha} - a\right) c =
  \left(\frac{\sqrt G \alpha}{\sqrt E} - a\right) b = 0.
\end{equation}
Hence, if $F = 0$ we get
\begin{equation}
  \label{eq:46}(a, b, c) =
  \begin{cases}
    \displaystyle \left(\frac{\sqrt E}{\sqrt G \alpha}, 0, \pm \sqrt{1
        - \frac{E}{G \alpha^2}}\right), & \displaystyle \text{if }
    \frac E G \le \alpha^2, \\[1.2pc]
    \displaystyle \left(\frac{\sqrt G \alpha}{\sqrt E}, \pm \sqrt{1 -
        \frac{G \alpha^2}{E}}, 0\right), & \displaystyle \text{if }
    \alpha^2 \le \frac E G.
  \end{cases}
\end{equation}

\subsubsection{The case of $J_2$}

The equations for the integrability of $J_2$ are somewhat more
delicate as they behave differently depending on the values of $r,
s$. The general form for such equations is
\begin{align}
  a {\left(\sqrt{r} - \sqrt{s}\right)}
  & = -\frac{{\left(1 - b^{2}\right)} {\left(1 - \sqrt{r s}\right)}
    \sqrt{\Delta}}{G} \label{eq:36}, \\
  a {\left(1 - \sqrt{r s}\right)}
  & = -\frac{{\left(1 - c^{2}\right)} {\left(\sqrt{r} -
    \sqrt{s}\right)}  \sqrt{\Delta}}{E} \label{eq:43}, \\
  F a {\left(\sqrt{r} - \sqrt{s}\right)}
  & = b c {\left(\sqrt{r} - \sqrt{s}\right)}
    \sqrt{\Delta} \label{eq:48}, \\
  F a {\left(1 - \sqrt{r s}\right)}
  & = b c {\left(1 - \sqrt{r s}\right)} \sqrt{\Delta} \label{eq:49}, \\
  F c {\left(1 - \sqrt{r s}\right)}
  & = - G b {\left(\sqrt{r} - \sqrt{s}\right)} -  a b
    {\left(1 - \sqrt{r s}\right)} \sqrt{\Delta} \label{eq:47}, \\
  F b {\left(\sqrt{r} - \sqrt{s}\right)}
  & = -E c {\left(1 - \sqrt{r s}\right)} -a c {\left(\sqrt{r} -
    \sqrt{s}\right)} \sqrt{\Delta} \label{eq:51},
\end{align}
where $\Delta = E G - F^2$ as in the previous case.

Recall that when $s = r = 1$ all the equations hold trivially, which
means that $J_2$ is a complex structure for all $(a, b, c) \in
S^2$. This is a result already known (see
\cite{abbena-1997,abbena-2001}). If $r = s < 1$, equations (\ref{eq:36})
to (\ref{eq:51}) reduce to
\begin{align}
  \label{eq:52}
  a = 0, && b = \pm 1, && c = 0.
\end{align}
So, we can assume that $0 < s < r < 1$. Let us denote
\begin{equation}
  \label{eq:53}
  \beta = \frac{\sqrt r - \sqrt s}{1 - \sqrt{r s}}
\end{equation}
and notice that $\beta$ is always positive. Now we can rewrite
equations (\ref{eq:36}) to (\ref{eq:51}) as
\begin{align}
  a & = - \frac{(1 - b^2) \sqrt \Delta}{G \beta} \label{eq:54}, \\
  a & = - \frac{(1 - c^2) \beta \sqrt \Delta}{E} \label{eq:55}, \\
  Fa & = b c \sqrt \Delta  \label{eq:56}, \\
  Fb & = - \frac{Ec}{\beta} - a c \sqrt \Delta \label{eq:57}, \\
  Fc & = - G \beta b - a b \sqrt \Delta \label{eq:58}.
\end{align}

Notice that equations (\ref{eq:54}) to (\ref{eq:58}) are formally
equal to equations (\ref{eq:37}) to (\ref{eq:41}) if we replace
$\alpha$ by $-\beta$. The only difference is that in this case $a < 0$
and $b, c$ have opposite signs. Since we never actually use the value
of $\alpha$ when solving (\ref{eq:37}) to (\ref{eq:41}) we can
conclude that when $F \neq 0$
\begin{align}
  a & = \frac12 \left( \frac{-1}{\sqrt \Delta} \left(\frac{E}{\beta} + G
        \beta \right) + \sqrt{\frac 1 \Delta \left(\frac E \beta  + G
          \beta\right)^2 - 4}\right), \\
  b & = \pm \sqrt{1 + \frac{G a \beta}{\sqrt \Delta}}, \\
  c & = \mp \sqrt{1 + \frac{E a}{\beta \sqrt \Delta}}
\end{align}
and if $F = 0$,
\begin{equation}
  (a, b, c) =
  \begin{cases}
    \displaystyle \left(- \frac{\sqrt E}{\sqrt G \beta}, 0, \pm \sqrt{1
        - \frac{E}{G \beta^2}}\right), & \displaystyle \text{if } \frac E G \le
    \beta^2 \\[1.2pc]
    \displaystyle \left(-\frac{\sqrt G \beta}{\sqrt E}, \pm \sqrt{1 -
        \frac{G \beta^2}{E}}, 0\right), & \displaystyle \text{if } \beta^2 \le
    \frac E G
  \end{cases}
\end{equation}

Summarizing, we obtained the following result.

\begin{theorem}
  \label{sec:case-j_2}
  Consider in $H_5$ the left-invariant metric $g$ given in
  (\ref{eq:33}). Then $(g, J)$ is an Hermitian structure on $H_5$ if
  and only if
  \begin{enumerate}
  \item $J = \pm J_1$, as in (\ref{eq:34}) with $a, b, c$ given as in
    Table \ref{tab:J1} or
  \item $J = \pm J_2$, as in (\ref{eq:35}) with $a, b, c$ given as in
    Table \ref{tab:J2}.
  \end{enumerate}
  In particular, every left-invariant metric on $H_5$ is Hermitian
  with respect to some left-invariant complex structure.
\end{theorem}

\begin{table}[ht]
  \caption{Case $\mathfrak h_5$: $J_1$ where $\Delta = E G - F^2$, $\alpha = \frac{\sqrt r +
      \sqrt s}{1 + \sqrt{r s}}$, $\gamma = \frac E \alpha + G \alpha$}
  \centering {\tabulinesep=1.2mm
    \begin{tabu}{|c|c|c|c|c|}
      \hline
      $a$ & $b$ & $c$ & $F$ & $(r, s)$ \\
      \hline\hline
      $\frac{\gamma - \sqrt{ {\gamma^2} - 4  \Delta}}{2 \sqrt \Delta}$
      & $ \pm\sqrt{1 - \frac{G a \alpha}{\sqrt \Delta}}$
      & $\pm\sqrt{1 - \frac{E a}{\alpha \sqrt \Delta}}$
      & $> 0$ & any  \\
      \hline
      $\frac{\sqrt E}{\sqrt G \alpha}$ & $0$
      & $\pm \sqrt{1 - \frac{E}{G \alpha^2}}$
      & $= 0$ & $\frac E G \le \alpha^2$ \\
      \hline
      $\frac{\sqrt G \alpha}{\sqrt E}$ & $\pm \sqrt{1 - \frac{G
          \alpha^2}{E}}$
      & $0$ & $= 0$ & $\alpha^2 \le \frac E G$ \\
      \hline
    \end{tabu}}
  \label{tab:J1}
\end{table}

\begin{table}[ht]
  \caption{Case $\mathfrak h_5$: $J_2$ where $\Delta = E G - F^2$, $\beta = \frac{\sqrt r -
      \sqrt s}{1 - \sqrt{r s}}$, $\delta = \frac E \beta + G \beta$}
  \centering
  {\tabulinesep=1.2mm
    \begin{tabu}{|c|c|c|c|c|}
      \hline
      $a$ & $b$ & $c$ & $F$ & $(r, s)$ \\
      \hline\hline
      \multicolumn{3}{|c|}{ $a^2 + b^2 + c^2 = 1$ } & any & $s = r = 1$\\
      \hline
      $0$ & $\pm 1$ & 0 & any & $s = r < 1$ \\
      \hline
      $-\frac{\delta - \sqrt{ {\delta^2} - 4  \Delta}}{2 \sqrt \Delta}$
          & $ \pm\sqrt{1 + \frac{G a \beta}{\sqrt \Delta}}$
                & $\mp\sqrt{1 + \frac{E a}{\beta \sqrt \Delta}}$
                      & $> 0$ & $s < r < 1$ \\
      \hline
      $-\frac{\sqrt E}{\sqrt G \beta}$ & $0$
                & $\pm \sqrt{1 - \frac{E}{G \beta^2}}$
                      & $= 0$ & $s < r < 1$ and $\frac E G \le \beta^2$ \\
      \hline
      $-\frac{\sqrt G \beta}{\sqrt E}$
          & $\pm \sqrt{1 - \frac{G \beta^2}{E}}$ & $0$ & $= 0$
                            & $s < r < 1$ and $\beta^2 \le \frac E G$ \\
      \hline
    \end{tabu}
  }
  \label{tab:J2}
\end{table}

\begin{proof}
  It follows from the above discussion. Notice that we introduce the
  $\pm$ sign in the statement of the theorem so that our
  classification also includes the Hermitian structures which reverse
  the orientation.
\end{proof}

\subsection{Hermitian structures on $\mathfrak h_4$}

We follow the same approach as in the previous case. Remember that for
$H_4$, the left-invariant metrics, up to isometric automorphism, are
given, in the standard basis $e_1, \ldots, e_6$ defined in Section
\ref{sec:case-mathfrak-h_4}, by
\begin{equation}
  \label{eq:61}
  g =
  \begin{pmatrix}
    1 & 0 & 0 & 0 & 0 & 0 \\
    0 & 1 & 0 & 0 & 0 & 0 \\
    0 & 0 & 1 & 0 & 0 & 0 \\
    0 & 0 & 0 & r & 0 & 0 \\
    0 & 0 & 0 & 0 & E & F \\
    0 & 0 & 0 & 0 & F & G
  \end{pmatrix}
\end{equation}
where $0 < r \le 1$, $E, G > 0$, $F \ge 0$ and
$\Delta = E G - F^2 > 0$. In this case the orientation preserving
almost Hermitian structures, with respect to this metric, are given by
\begin{equation}
  \label{eq:62}
  J_1 = \begin{pmatrix}
    0 & -a & -b & -c \sqrt r & 0 & 0 \\[.2pc]
    a & 0 & -c & b \sqrt r & 0 & 0 \\[.2pc]
    b & c & 0 & -a \sqrt r & 0 & 0\\[.2pc]
    \frac c {\sqrt r} & - \frac b {\sqrt r} & \frac a {\sqrt r} & 0 &
    0 & 0 \\[.2pc]
    0 & 0 & 0 & 0 & - \frac F {\sqrt \Delta} & - \frac G {\sqrt
      \Delta} \\[.4pc]
    0 & 0 & 0 & 0 & \frac E {\sqrt \Delta} & \frac F {\sqrt \Delta}
  \end{pmatrix}
\end{equation}
and
\begin{equation}
  \label{eq:63}
  J_2 = \begin{pmatrix}
    0 & -a & -b & -c \sqrt{r} & 0 & 0 \\[.2pc]
    a & 0 & c & -b \sqrt{r} & 0 & 0 \\[.2pc]
    b & -c  & 0 & a \sqrt{r} & 0 & 0 \\[.2pc]
    \frac{c}{\sqrt{r}} & \frac{b}{\sqrt{r}} & -\frac{a}{\sqrt{r}} & 0
    & 0 & 0 \\[.2pc]
    0 & 0 & 0 & 0 & \frac F {\sqrt \Delta} & \frac G {\sqrt
      \Delta} \\[.4pc]
    0 & 0 & 0 & 0 & - \frac E {\sqrt \Delta} & - \frac F {\sqrt
      \Delta}
  \end{pmatrix}
\end{equation}
where
\begin{equation}
  \label{eq:59}
  a^2 + b^2 + c^2 = 1.
\end{equation}
With the same ideas as in Subsection \ref{sec:hermitian-iwasawa}, we
can find conditions on $a, b, c$ for the integrability of $J_1$,
$J_2$. For the sake of completeness we present the non-trivial
equations but we omit the calculations that are too similar to the
ones in the case of $\mathfrak h_5$.

\subsubsection{The case of $J_1$}

Let us denote
\begin{equation}
  \label{eq:60}
  \alpha = \frac{1 + \sqrt r}{\sqrt r}.
\end{equation}
Then, the non-trivial equation for the vanishing of Nijenhuis tensor
of $J_1$ are
\begin{align*}
  b & = - \frac{(1 - a^2) \sqrt \Delta}{G \alpha}, \\
  b & = -\frac{(1 - c^2) \alpha \sqrt \Delta}{E}, \\
  F b & = - a c \sqrt \Delta, \\
  F a & = \frac{E c}{\alpha} + b c \sqrt \Delta, \\
  F c & = G \alpha a + a b \sqrt \Delta,
\end{align*}
which are the same equations as (\ref{eq:37}) to (\ref{eq:41}) under
the symmetry $(a, b, c) \mapsto (-b, a, c)$.

\subsubsection{The case of $J_2$}

The general form for the meaningful equations for the integrability of $J_2$ is
\begin{align*}
  b (1 - \sqrt r) & = - \frac{(1 - a^2) \sqrt r \sqrt \Delta}{G}, \\
  b & = - \frac{(1 - c^2) (1 - \sqrt r) \sqrt \Delta}{E \sqrt r}, \\
  F b & = - a c \sqrt \Delta, \\
  F a (1 - \sqrt r) & = E c \sqrt r + b c (1 - \sqrt r) \sqrt \Delta, \\
  F c & = \frac{G a (1 - \sqrt r)}{\sqrt r} + a b \sqrt \Delta,
\end{align*}
which resembles equations (\ref{eq:36}) to (\ref{eq:51}). Moreover,
if $r \neq 1$ and we denote
\begin{equation*}
  \beta = \frac{1 - \sqrt r}{\sqrt r},
\end{equation*}
then these equations simplify to
\begin{align*}
  b & = - \frac{(1 - a^2) \sqrt \Delta}{G \beta}, \\
  b & = - \frac{(1 - c^2) \beta \sqrt \Delta}{E}, \\
  F b & = - a c \sqrt \Delta, \\
  F a & = \frac{E c}{\beta} + b c \sqrt \Delta, \\
  F c & = G \beta a + a b \sqrt \Delta.
\end{align*}

Again, these equations behave exactly as in the previous case, after
replacing $\alpha$ with $\beta$. So, with no extra effort we achieve
the following classification of the left-invariant Hermitian
structures on $H_4$.

\begin{theorem}
  \label{sec:case-j_2-1}
  Consider in $H_4$ the left-invariant metric induced by $g$, given as
  in (\ref{eq:61}). Then $(g, J)$ is an Hermitian structure on $H_4$
  if and only if
  \begin{enumerate}
  \item $J = \pm J_1$, as in (\ref{eq:62}) with $a, b, c$ given as in
    Table \ref{tab:J1-h4} or
  \item $J = \pm J_2$, as in (\ref{eq:63}) with $a, b, c$ given as in
    Table \ref{tab:J2-h4}.
  \end{enumerate}
  In particular, every left-invariant metric on $H_4$ is Hermitian
  with respect to some left-invariant complex structure.
\end{theorem}

\begin{table}[ht]
  \caption{Case $\mathfrak h_4$: $J_1$ where $\Delta = E G - F^2$, $\alpha = \frac{1 + \sqrt
      r}{\sqrt r}$, $\gamma = \frac E \alpha + G \alpha$}
  \centering
  {\tabulinesep=1.2mm
    \begin{tabu}{|c|c|c|c|c|}
      \hline
      $a$ & $b$ & $c$ & $F$ & $r$ \\
      \hline\hline
      $ \pm\sqrt{1 + \frac{G b \alpha}{\sqrt \Delta}}$
      & $-\frac{\gamma - \sqrt{ {\gamma^2} - 4  \Delta}}{2 \sqrt \Delta}$
      & $\pm\sqrt{1 - \frac{E b}{\alpha \sqrt \Delta}}$
      & $> 0$ & any  \\
      \hline
      $0$ & $-\frac{\sqrt E}{\sqrt G \alpha}$
      & $\pm \sqrt{1 - \frac{E}{G \alpha^2}}$
      & $= 0$ & $\frac E G \le \alpha^2$ \\
      \hline
      $\pm \sqrt{1 - \frac{G \alpha^2}{E}}$
      & $-\frac{\sqrt G \alpha}{\sqrt E}$
      & $0$ & $= 0$ & $\alpha^2 \le \frac E G$ \\
      \hline
    \end{tabu}}
  \label{tab:J1-h4}
\end{table}

\begin{table}[ht]
  \caption{Case $\mathfrak h_4$: $J_2$ where $\Delta = E G - F^2$,
    $\beta = \frac{1 - \sqrt r}{\sqrt{r}}$,
    $\delta = \frac E \beta + G \beta$}
  \centering
  {\tabulinesep=1.2mm
    \begin{tabu}{|c|c|c|c|c|}
      \hline
      $a$ & $b$ & $c$ & $F$ & $r$ \\
      \hline\hline
      $\pm 1$ & $0$ & $0$ & any & $r = 1$ \\
      \hline $\pm\sqrt{1 + \frac{G b \beta}{\sqrt \Delta}}$ &
      $-\frac{\delta - \sqrt{ {\delta^2} - 4 \Delta}}{2 \sqrt \Delta}$
      & $\pm\sqrt{1 + \frac{E b}{\beta \sqrt \Delta}}$
      & $> 0$ & $0 < r < 1$ \\
      \hline $0$ & $-\frac{\sqrt E}{\sqrt G \beta}$ &
      $\pm \sqrt{1 - \frac{E}{G \beta^2}}$
      & $= 0$ & $0 < r < 1$ and $\frac E G \le \beta^2$ \\
      \hline $\pm \sqrt{1 - \frac{G \beta^2}{E}}$ &
      $-\frac{\sqrt G \beta}{\sqrt E}$ & $0$ & $= 0$
      & $0 < r < 1$ and $\beta^2 \le \frac E G$ \\
      \hline
    \end{tabu}
  }
  \label{tab:J2-h4}
\end{table}

\begin{remark}
  Recall that in $H_4$ there is, up to automorphism, a unique
  left-invariant abelian structure, the one given in the proof of
  Corollary \ref{sec:automorphism-group-3}, and it can be obtained as
  $J_2$ from (\ref{eq:63}) with $a = 1$, $b = c = 0$, $F = 0$ and
  $G = E$.
\end{remark}

\subsection{Hermitian structures on $\mathfrak h_6$}

The case of $\mathfrak h_6$ can be treated in the same way as
$\mathfrak h_5$ and $\mathfrak h_4$. And as a matter of fact,
calculations are much simpler for $\mathfrak h_6$. We do no repeat
such calculations but only state the theorem of classification of
Hermitian structures on $\mathfrak h_6$. Recall from Theorem
\ref{sec:left-invar-metr-2} that any left-invariant metric on $H_6$ is
equivalent to one, and only one, of the form
\begin{equation}
  \label{eq:50}
  g = \diag(1, 1, 1, 1, E, G),
\end{equation}
where $0 < E \le G$ and with respect to the standard basis
$e_1, \ldots, e_6$ given at the beginning of Section
\ref{sec:case-mathfrak-h_6}.

\begin{theorem}
  Consider in $H_6$ the left-invariant metric induced by $g$, as in
  (\ref{eq:50}) and let us denote $\alpha = \sqrt{E / G}$. Then
  $(g, J)$ is a Hermitian structure on $H_6$ if and only if, in the
  standard basis, $J = \pm J_1^\pm$ or $J = \pm J_2^\pm$ where
  \begin{equation*}
    J_1^\pm =
    \begin{pmatrix}
      0 & 0 & \pm \sqrt{1 - \alpha^2} & - \alpha & 0 & 0 \\
      0 & 0 & - \alpha & \mp \sqrt{1 - \alpha^2} & 0 & 0 \\
      \mp \sqrt{1 - \alpha^2} & \alpha & 0 & 0 & 0 & 0 \\
      \alpha & \pm \sqrt{1 - \alpha^2} & 0 & 0 & 0 & 0 \\
      0 & 0 & 0 & 0 & 0 & - \frac 1 \alpha \\
      0 & 0 & 0 & 0 & \alpha & 0
    \end{pmatrix}
  \end{equation*}
  and
  \begin{equation*}
    J_2^\pm =
    \begin{pmatrix}
      0 & 0 & \pm \sqrt{1 - \alpha^2} & - \alpha & 0 & 0 \\
      0 & 0 & \alpha & \pm \sqrt{1 - \alpha^2} & 0 & 0 \\
      \mp \sqrt{1 - \alpha^2} & - \alpha & 0 & 0 & 0 & 0 \\
      \alpha & \mp \sqrt{1 - \alpha^2} & 0 & 0 & 0 & 0 \\
      0 & 0 & 0 & 0 & 0 &  \frac 1 \alpha \\
      0 & 0 & 0 & 0 & - \alpha & 0
    \end{pmatrix}.
  \end{equation*}
  In particular, $J_1^+ = J_1^-$ and $J_2^+ = J_2^-$ if and only if
  $E = G$.
\end{theorem}

\begin{corollary}
  Every left-invariant metric on $H_6$ is Hermitian.
\end{corollary}

\subsection{Hermitian structures on $\mathfrak h_2$}

The cases of $\mathfrak h_2$ and $\mathfrak h_9$ are significantly
harder to treat by the above methods. In fact, the generic metrics in
our classification for both Lie algebras are not diagonal, with
respect to the standard basis, when restricted to the orthogonal
complement of the commutator. So the polynomial equations describing
the integrability of an arbitrary almost Hermitian structure, by means
of the Cholesky decomposition, become wild. However, one can still
recover some information from these polynomials in order to estimate
the amount of Hermitian structures.

In order to simplify the exposition, let us change slightly the
notation for the metrics computed in Theorem
\ref{sec:left-invar-metr-5} for the Lie group $H_2$. We denote such
metrics by
\begin{equation}
  \label{eq:71}
  g = \sum_{i = 1}^4 e^i \otimes e^i + 2 A e^1 \otimes e^3 + 2 B
  e^2 \otimes e^4 + E e^5 \otimes e^5 + 2F e^5 \otimes e^6 + G e^6
  \otimes e^6,
\end{equation}
where $0 \le A \le B < 1$, $E, F, G \ge 0$ and $EG - F^2 > 0$. As in
the previous cases, we denote $\Delta = EG - F^2$ and in addition we
introduce the notations
\begin{align*}
  \alpha = \sqrt{1 - A^2}, && \beta = \sqrt{1 - B^2}.
\end{align*}
Observe that these parameters satisfy
\begin{equation*}
  A^2 + \alpha^2 = B^2 + \beta^2 = 1.
\end{equation*}
In order to simplify some calculations, we further define
\begin{align*}
  \phi = B \alpha - A \beta, && \psi = A B + \alpha \beta.
\end{align*}
It is easy to see that $\psi$ is always positive and $\phi$ is non
negative, and positive if $A \neq B$.

The orientation-preserving almost Hermitian structures with respect to
(\ref{eq:71}) are homeomorphic to the disjoint union of two
$2$-spheres. Since the computation in the general case become very
complicated, we illustrate the procedure only for the connected component
given by
\begin{equation}
  \label{eq:72}
  J =
  \begin{pmatrix}
    -\frac{A b}{\alpha} & -\frac{a \alpha + A c}{\alpha} &
    -\frac{b}{\alpha} & -\frac{a \phi + c \psi}{\alpha} & 0 & 0
    \\[0.4pc]
    \frac{a \beta - B c}{\beta} & \frac{B b}{\beta} & -\frac{a \phi
      + c \psi}{\beta} & \frac{b}{\beta} & 0 & 0 \\[0.4pc]
    \frac{b}{\alpha} & \frac{c}{\alpha} & \frac{A b}{\alpha} &
    -\frac{a \beta - B c}{\alpha} & 0 & 0 \\[0.4pc]
    \frac{c}{\beta} & -\frac{b}{\beta} & \frac{a \alpha + A
      c}{\beta} & -\frac{B b}{\beta} & 0 & 0 \\[0.4pc]
    0 & 0 & 0 & 0 & -\frac{F}{\sqrt{\Delta}} &
    -\frac{G}{\sqrt{\Delta}} \\[0.4pc]
    0 & 0 & 0 & 0 & \frac{E}{\sqrt{\Delta}} &
    \frac{F}{\sqrt{\Delta}}
  \end{pmatrix}
\end{equation}
where
\begin{equation}
  \label{eq:75}
  a^2 + b^2 + c^2 = 1.
\end{equation}

After some long computations, which were verify using the computer
software SageMath, we obtain that $J$ as in (\ref{eq:72}) is
integrable if the following equations hold:

\begin{align}
  0 & = -a^{2} \beta \phi  + b^{2} A + c^{2} B \psi +  a c {\left(B \phi - \beta
      \psi\right)} - \frac{b {\left(F \alpha + G \beta\right)}}{
      \sqrt{\Delta}}, \\
  0 & = a b \sqrt{\Delta} + a E \phi+ c {\left(E \psi + F\right)}, \\
  0 & =  (1 - a^{2}) \sqrt{\Delta} + b E \phi \label{eq:73}, \\
  0 & = (a \phi + c \psi)^2 + b^{2} + \frac{G b \phi}{\sqrt{\Delta}}, \\
  0 & = a c \alpha + (1 - a^{2}) A - \frac{b {\left(F \alpha + E
      \beta\right)}}{\sqrt{\Delta}}, \\
  0 & = a c \phi + (1 - a^{2}) \psi - \frac{b F
      \phi}{\sqrt{\Delta}} \label{eq:74}, \\
  0 & = (c \phi - a \psi) b \sqrt{\Delta} + a F \phi + c {\left(F \psi
      + G\right)}, \\
  0 & = a^{2} \alpha \phi + b^{2} B + c^{2}  A \psi + a c {\left(A \phi + \alpha
      \psi\right)} + \frac{b {\left(G \alpha + F
      \beta\right)}}{\sqrt{\Delta}}, \\
  0 & = a c \beta  - {\left(1 - a^{2}\right)} B - \frac{b {\left(E \alpha + F
      \beta\right)}}{\sqrt{\Delta}}.
\end{align}

We can fully solve the following important particular case.
\begin{proposition}
  If $A = B$, or equivalently $\phi = 0$, then $J$ is integrable if and
  only if $a = \pm 1$ and $b = c = 0$. Moreover, in this case $J$ is abelian.
\end{proposition}

\begin{remark}
  We can approach the abelian case from the classification given in
  \cite{Andrada_2011} using a similar argument as the one we will use
  in the next subsection. In fact, one can prove that if an abelian
  structure is hermitian with respect to a metric of the form
  (\ref{eq:71}), then $A = B$.
\end{remark}

One can argue that the same approach mentioned in the above remark
could be use in the non-abelian case by using the classification of
\cite{ceballos-2014}, but due to the complexity of the problem,
the calculations in this case turn extremely difficult to be solve
explicitly. Instead, to treat the case $\phi \neq 0$ we note that from
(\ref{eq:73}),
\begin{equation*}
  b = -\frac{(1 - a^2) \sqrt \Delta}{E \phi}.
\end{equation*}
Observe that, since $a \neq \pm 1$, we get that $b$ is negative and from
(\ref{eq:74}) we get
\begin{align*}
  c & = \frac{1}{a \alpha}\left( \frac{F b \phi}{\sqrt \Delta} - (1 -
      a^2)\psi \right) \\
    & = - \frac{1 - a^2}{a \alpha} \left( \frac{F}{E \phi} + \psi \right).
\end{align*}

By replacing these values of $b$ and $c$ in (\ref{eq:75}) we obtain a
quartic equation for $a$, which is in fact quadratic in
$a^2$. Disregarding the complex solutions we get the following result.

\begin{proposition}
  There exists at most two different values of $(a, b, c)$ such that
  $J$ is a complex structure.
\end{proposition}

Note that one still needs to check that the almost complex structure
given by this construction satisfies the integrability equations other
that (\ref{eq:73}) and (\ref{eq:74}). However, it is our belief that
these equations always admit a solution, as we could check by solving
them numerically for random parameters.

\subsection{Hermitian structures on $\mathfrak h_9$}

As in the above case, the problem of determining the Hermitian metrics
on $H_9$ is very hard, since the moduli space of left-invariant
metrics is described by six real parameters. However, we can perform a
qualitative analysis by using an appropriate decomposition of a
subgroup of automorphisms which preserve the metrics given in
(\ref{eq:32}). At some point, calculations become extremely tedious and
we use SageMath to check some computations.

Recall that in the basis $\hat e_1, \ldots, \hat e_6$, every
left-invariant metric on $H_9$ is equivalent, via an automorphism, to
one in the slice $\Phi(\Sigma)$, which consist of all the metrics
induced by the matrices given in (\ref{eq:70}), where $A, B, C >0$ and
$D, E, F \in \mathbb R$. Notice also that form \cite{Andrada_2011}
there exists a unique complex structure (which is also abelian) on
$\mathfrak h_9$ up to automorphism. After rearranging the basis, we
get that such complex structure, say $J_0$, is given by the relations
\begin{align*}
  J_0 \hat e_1 = - \hat e_2,
  && J_0 \hat e_3 = \hat e_5,
  && J_0 \hat e_4 = - \hat e_6.
\end{align*}

Given a complex structure $J$ on $H_9$ and
$\varphi \in \Aut(\mathfrak h_9)$ we denote by $\varphi \cdot J =
\varphi J \varphi^{-1}$ the standard action of the automorphism group
on the space on complex structures. Let
\begin{equation*}
  \Aut_0(\mathfrak h_9)^\Sigma = \{\varphi \in \Aut_0(\mathfrak h_9):
  (g, \varphi \cdot J_0) \text{ is an Hermitian structure for some } g
  \in \Phi(\Sigma)\}.
\end{equation*}
As we will see later, $\Aut_0(\mathfrak h_9)^\Sigma$ is not a subgroup of
$\Aut_0(\mathfrak h_9)$, but this set is relevant to understand the
problem of the existence of Hermitian metrics. In fact, since every complex
structure has the form $\varphi \cdot J_0$, if $g$ is a Hermitian
metric then there must exist $g' \in \Phi(\sigma)$ and $\varphi' \in
\Aut_0(\mathfrak h_9)^\Sigma$ such that $g'$ is isometric to $g$ and
$(g', \varphi' \cdot J_0)$ is a Hermitian structure.

Lets keep the notation of (\ref{eq:30}) for a generic
$\varphi \in \Aut_0(\mathfrak h_9)$. Even though
$\Aut_0(\mathfrak h_9)^\Sigma$ is not a group, it is not hard to see
that it is contained in
\begin{equation*}
  G = \{\varphi \in \Aut_0(\mathfrak h_9): a_{21} = 0, a_{22} =
  a_{11}\},
\end{equation*}
which is a closed subgroup of the automorphism group. The group
structure is easier to understand than the one of the full
automorphism group. In fact, consider the following subgroups of $G$:
\begin{align*}
  G_1 & = \left\{
        \begin{pmatrix}
          1 & 0 & 0 & 0 & 0 & 0 \\
          0 & 1 & 0 & 0 & 0 & 0 \\
          0 & 0 & 1 & 0 & 0 & 0 \\
          0 & 0 & 0 & 1 & 0 & 0 \\
          a_{51} & a_{52} & 0 & 0 & 1 & 0 \\
          a_{61} & a_{62} & a_{63} & a_{64} & -a_{52} & 1
        \end{pmatrix}
                                                        \right\},
  \\
  G_2 & = \left\{
        \begin{pmatrix}
          1 & 0 & 0 & 0 & 0 & 0 \\
          0 & 1 & 0 & 0 & 0 & 0 \\
          a_{31} & a_{32} & 1 & 0 & 0 & 0 \\
          a_{41} & a_{42} & a_{43} & 1 & 0 & 0 \\
          0 & 0 & 0 & 0 & 1 & 0 \\
          0 & 0 & 0 & 0 & a_{31} & 1
        \end{pmatrix}
                                   \right\}, \\
  G_3 & = \{\diag(a_{11}, a_{11}, a_{11}^2, a_{44}, a_{11}^2,
        a_{11}^3)\}.
\end{align*}

With a routine calculation, one can see that $G_1$ is a normal
subgroup of $G$. The group $G_2$ is not a normal subgroup of $G$,
however, $G_1 \rtimes G_2$ is a normal subgroup of $G$ and therefore
$G$ is the triple semi-direct product
\begin{equation*}
  G = (G_1 \rtimes G_2) \rtimes G_3.
\end{equation*}

Let us denote
\begin{equation*}
  G_i^{\Sigma} = G_i \cap \Aut_0(\mathfrak h_9)^\Sigma
\end{equation*}
for $i = 1, 2, 3$ and
\begin{equation*}
  \Phi(\Sigma_i) = \{g \in \Phi(\Sigma): (g, \varphi \cdot J_0) \text{
  is a Hermitian structure for some }\varphi \in G_i\}.
\end{equation*}

The following results are straightforward.

\begin{lemma}
  \label{sec:herm-struct-mathfr}
  \begin{enumerate}
  \item
    \begin{equation*}
      \Phi(\Sigma_1) = \left\{
        \begin{pmatrix}
          1 & 0 & 0 & 0 & 0 & 0 \\
          0 & 1 & 0 & 0 & 0 & 0 \\
          0 & 0 & A^{2} & 0 & 0 & 0 \\
          0 & 0 & 0 & E^{2} + 1 & \sqrt{E^{2} + 1} A E & 0 \\
          0 & 0 & 0 & \sqrt{E^{2} + 1} A E & {\left(E^{2} + 1\right)} A^{2} & 0 \\
          0 & 0 & 0 & 0 & 0 & E^{2} + 1
        \end{pmatrix}: A > 0, E \in \mathbb R\right\}.
    \end{equation*}
  \item
    \begin{equation*}
      G_1^\Sigma = \left\{\varphi =
        \begin{pmatrix}
          1 & 0 & 0 & 0 & 0 & 0 \\
          0 & 1 & 0 & 0 & 0 & 0 \\
          0 & 0 & 1 & 0 & 0 & 0 \\
          0 & 0 & 0 & 1 & 0 & 0 \\
          0 & 0 & 0 & 0 & 1 & 0 \\
          0 & 0 & a_{63} & 0 & 0 & 1
        \end{pmatrix}: a_{63} \in \mathbb R
      \right\}
    \end{equation*}
    is a subgroup of $G$. Moreover, if $\varphi \in G_1^\Sigma$, then
    $\varphi \cdot J_0$ is Hermitian with respect to any metric in
    $\Phi(\Sigma_1)$ such that
    $a_{63} = \frac{A E}{\sqrt{E^{2} + 1}}$.
  \end{enumerate}
\end{lemma}

\begin{lemma}
\label{sec:herm-struct-mathfr-1}
  \begin{enumerate}
  \item
    \begin{equation*}
      \Phi(\Sigma_2) = \left
        \{
        \begin{pmatrix}
          1 & 0 & 0 & 0 & 0 & 0 \\
          0 & 1 & 0 & 0 & 0 & 0 \\
          0 & 0 & A^{2} & 0 & 0 & 0 \\
          0 & 0 & 0 & 1 & 0 & 0 \\
          0 & 0 & 0 & 0 & A^{2} + F^{2} & F \\
          0 & 0 & 0 & 0 & F & 1
        \end{pmatrix}: A > 0, F \in \mathbb R\right\}.
    \end{equation*}
  \item
    \begin{equation*}
      G_2^\Sigma = \left\{\varphi =
        \begin{pmatrix}
          1 & 0 & 0 & 0 & 0 & 0 \\
          0 & 1 & 0 & 0 & 0 & 0 \\
          0 & 0 & 1 & 0 & 0 & 0 \\
          0 & 0 & a_{43} & 1 & 0 & 0 \\
          0 & 0 & 0 & 0 & 1 & 0 \\
          0 & 0 & 0 & 0 & 0 & 1
        \end{pmatrix}: a_{43} \in \mathbb R
      \right\}
    \end{equation*}
    is a subgroup of $G$. Moreover, if $\varphi \in G_2^\Sigma$, then
    $\varphi \cdot J_0$ is Hermitian with respect to any metric in
    $\Phi(\Sigma_2)$ such that $a_{43} = -F$.
  \end{enumerate}
\end{lemma}

\begin{lemma}
  \label{sec:herm-struct-mathfr-2}
  \begin{enumerate}
  \item
    $\Phi(\Sigma_3) = \left\{\diag(1,1,A^2,1,A^2,C^2): A, C > 0
    \right\}$.
  \item $G_3^\Sigma = G_3$. Moreover, if $\varphi \in G_3$ then
    $\varphi \cdot J_0$ is Hermitian with respect to any metric in
    $\Phi(\Sigma_3)$ such that $C = \frac{a_{44}}{a_{11}^3}$.
  \end{enumerate}
\end{lemma}

\begin{remark}
  \label{sec:herm-struct-mathfr-3}
  It follows from Lemmas \ref{sec:herm-struct-mathfr} and
  \ref{sec:herm-struct-mathfr-1} that $G_1^\Sigma$ commutes with
  $G_2^\Sigma$ and their product is normalized by $G_3$. Moreover, the
  four dimensional Lie group
  \begin{equation*}
    G' = (G_1^\Sigma \times G_2^\Sigma) \rtimes G_3
  \end{equation*}
  is contained in $\Aut_0(\mathfrak h_9)^\Sigma$. In fact, an arbitrary
  element in $G'$ has the form
\begin{equation*}
  \varphi' =
      \begin{pmatrix}
        a_{11} & 0 & 0 & 0 & 0 & 0 \\
        0 & a_{11} & 0 & 0 & 0 & 0 \\
        0 & 0 & a_{11}^{2} & 0 & 0 & 0 \\
        0 & 0 & a_{43} & a_{44} & 0 & 0 \\
        0 & 0 & 0 & 0 & a_{11}^{2} & 0 \\
        0 & 0 & a_{63} & 0 & 0 & a_{11}^{3}
      \end{pmatrix}
    \end{equation*}
  where $a_{11}, a_{44} > 0$ and $a_{43}, a_{63} \in \mathbb R$, and
  after long computations one can check that $(g', \varphi' \cdot
  J_0)$ is a Hermitian structure for the metric $g' \in \Phi(\Sigma)$
  given by
  \begin{align*}
    B & = \frac{A^2 a_{11}^5}{\sqrt{A^2 a_{11}^{10} - a_{44}^2
        a_{63}^2}},
    && C = \frac{A a_{11}^{2} a_{44}}{\sqrt{A^2 a_{11}^{10} - a_{44}^2 a_{63}^2}},
    && D = 0, \\
    E &= \frac{a_{44} a_{63}}{\sqrt{A^{2} a_{11}^{10} - a_{44}^{2}
        a_{63}^{2}}},
    && F = -\frac{A a_{11}^{3} a_{43}}{\sqrt{A^2 a_{11}^{10} -
       a_{44}^2 a_{63}^2}}
    &&
  \end{align*}
  for sufficiently large values of $A$. However, notice that $G' \neq
  \Aut_0(\mathfrak h_9)^\Sigma$. In order to prove this we can note
  that the automorphisms in
  \begin{equation*}
    G'' = \left\{\varphi_{s, t} =
    \begin{pmatrix}
      1 & 0 & 0 & 0 & 0 & 0 \\
      0 & 1 & 0 & 0 & 0 & 0 \\
      0 & s & 1 & 0 & 0 & 0 \\
      t & -t & 0 & 1 & 0 & 0 \\
      s & 0 & 0 & 0 & 1 & 0 \\
      t & t & 0 & 0 & 0 & 1
    \end{pmatrix}: s, t \in \mathbb R
    \right\}
  \end{equation*}
  form an abelian group which is contained in
  $\Aut_0(\mathfrak h_9)^\Sigma$ such that $G'' \cap G' = \{I_6\}$. In
  fact, it is easy to see that $\varphi_{s, t} \cdot J_0$ is hermitian
  with respect to any metric of the form $g = \diag(1, 1, A, 1, A, 1)$
  (recall that all of these metrics belong to $\Phi(\Sigma_3)$, in the
  notation of Lemma \ref{sec:herm-struct-mathfr-2}).
\end{remark}

\begin{remark}
  $\Aut_0(\mathfrak h_9)^\Sigma$ is not a subgroup of
  $\Aut_0(\mathfrak h_9)$. In fact, one can check by a direct
  calculation that if $\varphi_2 \in G_2^\Sigma$ is not the identity,
  then $\varphi_{s, t} \varphi_2 \in \Aut_0(\mathfrak h_9)^\Sigma$ if
  and only if $s = 0$ (we are keeping the notation of Remark
  \ref{sec:herm-struct-mathfr-3}).
\end{remark}

Finally, using the previous results, we were able to detect a family of
left-invariant metrics on $H_9$ which are not Hermitian. 

\begin{proposition}
  Let us consider the left-invariant metrics on $H_9$ given by
\begin{equation*}
  g_{A, B}  = \hat e^1 \otimes \hat e^1 + \hat e^2 \otimes \hat e^2
  + A^2 \hat e^3 \otimes \hat e^3 + \hat e^4 \otimes \hat e^4 + B^2
  \hat e^5 \otimes \hat e^5 + \hat e^6 \otimes \hat e^6
  \end{equation*}
  with $A, B > 0$. Then $g_{A, B}$ is Hermitian if and only if
  $B = A$.
\end{proposition}

\begin{proof}
It follows form a direct computation. Notice that one only needs to deal with complex structures obtained by conjugating $J_0$ by an element of the group $G$.
\end{proof}

\end{document}